\documentclass[11pt]{amsart}
\usepackage{bm}
\usepackage{amsmath}
\usepackage{amsthm,amssymb,mathrsfs}
\usepackage{epic,eepic}
\usepackage{paralist,enumerate}
\usepackage[all]{xy}
\def\Log{{\rm Log}}
\def\aff{{\rm aff}}

\usepackage{tikz-cd}
\usepackage{xcolor,graphicx}
\usepackage[mathscr]{euscript}
\usepackage{circuitikz}
\ctikzset{bipoles/resistor/height=0.15}
\ctikzset{bipoles/resistor/width=0.4}

\usepackage{subcaption}

\definecolor{cadmiumgreen}{rgb}{0.0, 0.42, 0.24}
\usepackage[
colorlinks, citecolor=cadmiumgreen,
pdfauthor={Farhad Babaee},
pdftitle={ },
pdfstartview ={FitV},
]{hyperref}

\usepackage[
alphabetic,
msc-links,
nobysame,
lite,
]{amsrefs}

\usepackage{tikz, float}
\usepackage{tikz-cd}
\usetikzlibrary {positioning}

\usetikzlibrary{calc,decorations.markings}
\usetikzlibrary{shapes,snakes}

\usepackage[left=3.5cm,top=3.5cm,right=3.5cm]{geometry}
\DeclareSymbolFont{extraup}{U}{zavm}{m}{n}
\DeclareMathSymbol{\varheart}{\mathalpha}{extraup}{86}
\DeclareMathSymbol{\vardiamond}{\mathalpha}{extraup}{87}

\theoremstyle{definition}

\newtheorem{theorem}{Theorem}[section]
\newtheorem{definition}[theorem]{Definition}
\newtheorem{lemma}[theorem]{Lemma}
\newtheorem{proposition}[theorem]{Proposition}
\newtheorem{corollary}[theorem]{Corollary}
\newtheorem*{theorem*}{Theorem}

\usepackage{comment}
\newtheorem{remark}[theorem]{Remark}
\newtheorem{example}[theorem]{Example}

\newcommand{\R}{\mathbb{R}}

\newcommand{\Z}{\mathbb{Z}}
\newcommand{\lto}{\longrightarrow}
\newcommand{\lmto}{\longmapsto}
\newcommand{\CC}{\mathbb{C}}
\newcommand{\C}{\mathbb{C}}

\usepackage{bbm}

\newcommand{\trop}{\mathrm{trop}}

\newcommand{\cD}{\mathcal{D}}
\newcommand{\sD}{\mathscr{D}}
\newcommand{\sC}{\mathscr{C}}
\newcommand{\sU}{\mathscr{U}}

\newcommand{\cO}{\mathcal{O}}

\newcommand{\cC}{\mathcal{C}}
\newcommand{\sT}{\mathscr{T}}
\newcommand{\sR}{\mathscr{R}}
\newcommand{\sS}{\mathscr{S}}

\newcommand{\sW}{\mathscr{W}}
\newcommand{\sV}{\mathscr{V}}

\newcommand{\fq}{\mathfrak{q}}

\newcommand{\Trop}{\mathrm{Trop}}

\usepackage{hyperref}
\hypersetup{
  colorlinks=true,
  linkcolor=blue,
  filecolor=magenta,
  urlcolor=cyan,
  pdftitle={Overleaf Example},
  pdfpagemode=FullScreen,
  }

\usetikzlibrary{graphs,quotes}

 \newcommand{\supp}[1]{{\mathrm{supp}(#1)}}

\usepackage{scalerel}

\newtheorem{itheorem}{Theorem}
\newtheorem{iproblem}[itheorem]{Problem}
\newtheorem{iproposition}[itheorem]{Proposition}

\begin{document}
\title{CONTINUITY OF THE SUPERPOTENTIALS AND SLICES OF TROPICAL CURRENTS}
\author{Farhad Babaee}
\author{Tien Cuong Dinh}
\address{School of Mathematics, University of Bristol}
\email{farhad.babaee@bristol.ac.uk}
\address{Department of Mathematics, National University of Singapore}
\email{matdtc@nus.edu.sg}
\date{}
\begin{abstract}
We study the question of the continuity of slices of currents and explain how it relates to several seemingly unrelated problems in tropical geometry. On the one hand, through this lens, we show that the continuity of superpotentials constitutes a very general instance of stable intersection theory in tropical geometry. On the other hand, questions concerning tropicalisation with respect to a non-trivial valuation and the commutativity of tropicalisation with intersections offer insights to the problem of the continuity of slices of currents converging to a tropical current. Within our framework, we reformulate and reprove known theorems in tropical geometry and derive new results and techniques in both tropical geometry and complex dynamical systems.
\end{abstract}
\maketitle
\setcounter{tocdepth}{1}
\tableofcontents
\section{Introduction}
Let $X$ be a complex manifold of dimension $d,$ and $p,q$ non-negative integers with $d= p+q.$ We denote by $\cC^{q}(X)= \cC_{p}(X)$ the cone of positive closed bidegree $(q,q)$, or bidimension $(p,p)$-currents on $X.$ We also consider $\sD^q(X) = \sD_p(X),$ the $\R$-vector space spanned by $\sC^q(X).$ It is well-known that the intersection of two positive closed currents is not always defined. The main initial progress was due to the works of Federer \cite{Federer}. Federer defined the generic slicing theory of currents: for a dominant holomorphic map $f: X \lto Y$, and a positive closed current $\sT\in \cC_p(X) $, or more generally, a \emph{flat current}, a \emph{slice}
$$
\sT\wedge [f^{-1}(y)]   
$$
is well-defined for a generic $y \in Y.$ A different but approach was developed by Bedford and Taylor, and extended by Demailly \cite{Dem-analytic1}, and Fornaess--Sibony \cite{Fornaess-Sibony} which assert the following: if $\sS= dd^c u$ is a bidegree $(1,1)$-current, then 
$$
\sS \wedge \sT := dd^c(u \sT),
$$
can be defined or it is \emph{admissible}, when 
\begin{itemize}
    \item The potential, $u$, is bounded;
    \item $u$ is unbounded, but its unbounded locus has a small intersection with $\supp{\sT}.$
\end{itemize}
For instance, when $\sS = dd^c \log|f|$ and $\sT$ are two integration currents, such that their supports 
intersect in the expected dimension, then 
$$
\sS \wedge \sT = \sum c_i [C_i],
$$
where each $C_i$ is an irreducible component of the intersection, and $c_i$ is the corresponding vanishing number. This intersection coincides with the slicing of integration currents. 
\vskip 2mm
Demailly in \cite{Dem-gaz} asked the question of generalising the intersection theory to the case where $\sT$ is of a higher bidegree. In several works, the second-named author and Sibony introduced \emph{superpotential theory} and \emph{density} of currents to answer this question. In \cite{Dinh-Sibony:superpot}, the situation where $X$ is a homogeneous space was completely resolved, and in \cite{Dinh-Sibony:Kahler} the authors investigated the intersection theory for currents with continuous superpotentials on K\"ahler manifolds, which generalises the case of bounded potentials in bidegree $(1,1)$. In this article, we adopt the superpotential theory and the slicing theory. Let us also note that the works of Anderson, Eriksson, Kalm, Wulcan and Yger \cite{Yger-et-al}, and \cite{Anderson--Kalm--Wulcan} also consider a non-proper intersection theory.
\vskip 2mm
 Once the intersections are defined, one can ask the following \emph{continuity problem}:
\begin{iproblem}\label{prob:cont}
    Let $\sT_k$ be a sequence of positive closed currents on $X$ converging to $\sT.$  Let $\sS$ be also a positive closed current on $X.$ Find sufficient conditions such that 
    $$
     \lim_{k \to \infty} ( \sS \wedge \sT_k) =  \sS \wedge (\lim_{k\to \infty} \sT_k).
    $$
\end{iproblem}

We say that a current $\sS$ is on a compact K\"ahler manifold has a \emph{continuous superpotential}, if for any current $\sT$, the wedge product
$$
\sS \wedge \sT := \lim_{n \to \infty} (\sS \wedge \sT_n),
$$
is independent of the choice of smooth approximation $\sT_n \lto \sT$ (see Definition \ref{def:CSP}). Employing  results from \cite{Dinh-Sibony:Kahler}, we can partially answer Problem \ref{prob:cont}: 
\begin{iproposition}[See Proposition \ref{prop:slicing-CSP}]\label{iprop:slicing-CSP}
   Let $X$ be a compact K\"ahler manifold, $\sT_n \lto \sT$ be a convergent sequence in $\sD^p(X).$ If a current $\sS$ has a continuous superpotential, then 
    $$
    \sS \wedge \sT_n \lto \sS \wedge \sT. 
    $$
\end{iproposition} 
Problem~\ref{prob:cont} becomes more subtle when the current $\sS$ is an integration current. For instance, when $\sS$ and $\sT$ are integration currents and have improper intersections, we cannot, in general, find $\sT_n\lto \sT$ and hope that the slices converge. 
\vskip 2mm
In this article, we examine the problem of slicing with integration currents in the specific case where $\lim_{k \to \infty} \sT_k$  is a \emph{complex tropical current} and $S$ a complex torus.
The complex tropical currents (see \cite{Babaee, BH}) are closed currents on complex tori $(\C^*)^d$ or on a toric variety associated to a \emph{tropical cycle}. Recall that a tropical cycle is a weighted polyhedral complex satisfying the \emph{balancing condition} (see Definition \ref{BalancingCondition}). For a tropical cycle $\cC \subseteq \R^d,$ of dimension $p,$ the associated tropical current $\sT_{\cC} \in \sD_p((\C^*)^d),$ is a closed current with support $\Log^{-1}(\cC),$ where 
$$
\Log: (\C^*)^d \lto \R^d, \quad (z_1, \dots , z_d) \lmto (-\log|z_1|, \dots , -\log|z_d|).
$$
The tropical current $\sT_{\cC}$ can be naturally presented as a locally fibration of $\Log^{-1}(\cC),$ and we say $\cC$ is compatible with the fan $\Sigma,$ if the fibres of $\overline{\sT}_{\cC}$ intersect the toric invariant divisors of the toric variety $X_{\Sigma}$ transversely. Here $\overline{\sT}_{\cC}$ denotes the extension by zero of $\sT_{\cC}$ to the toric variety $X_{\Sigma}.$

\begin{itheorem}[See Theorem \ref{thm:trop-curr-CSP}]\label{ithm:trop-curr-CSP}
    Let $X_{\Sigma}$ be a smooth projective toric variety, and let $\cC$ be a tropical cycle compatible with $\Sigma.$ Then 
    $\overline{\sT}_{\cC}$ has a continuous superpotential. 
\end{itheorem}
The preceding theorem allows us to define the intersection product of a tropical current with any current on a compatible toric variety. We can then restrict this intersection product to the complex torus \( T_N \subseteq X_{\Sigma} \) and, using the isomorphism \( T_N \simeq (\mathbb{C}^*)^d \), define the intersection product of two tropical currents in \( (\mathbb{C}^*)^d \).
\vskip 2mm
On the tropical geometry side, there exists a \emph{stable intersection theory} for tropical cycles. The term “stable” refers precisely to the continuity of the intersection under generic translations of tropical cycles. This property is reminiscent of the \emph{fan displacement rule} in toric geometry (see \cite{Fulton--Sturmfels}), and more generally, of the \emph{moving lemma} and \emph{dynamic intersections} in algebraic geometry; see \cite{Fulton_int}. Following \cite{Fulton--Sturmfels}, the concept of stable intersection theory was studied in \cite{Kazarnovski}, \cite{RGST} and \cite{Mikh-trop-appli}, and extended in \cite{Jensen-Yu}. A new but equivalent approach was developed by \cite{All-Rau}. In the context of supercurrents, Lagerberg \cite{Lagerberg} derived an intersection theory for hypersurfaces. We also refer the reader to \cite{Shaw-Int-Mat, Gubler, Mihatsch} for related developments.
\vskip 2mm
\noindent
In the framework of this article, given Theorem \ref{ithm:trop-curr-CSP}, the stable intersection theory is a very special case of Proposition \ref{iprop:slicing-CSP}. 
\vskip 2mm
With the stable intersection and natural addition of tropical cycles, we have the $\mathbb{Q}$-algebra of the tropical cycles. 
\begin{itheorem}[See Theorem \ref{thm:Q-alg}]\label{thm:into_ring_iso}
   The assignment $ \cC \lmto \sT_{\cC} $ induces a $\mathbb{Q}$-algebra  isomorphism between
   \begin{itemize}
        \item [(a)] The $\mathbb{Q}$-algebra of tropical currents on $(\C^*)^d$ with the usual addition of currents and the wedge product.
        \item [(b)] The $\mathbb{Q}$-algebra generated by the tropical cycles in $\R^d$ with the natural addition  (Definition \ref{def:add-trop}) and stable intersection (Definition \ref{def:stable-int}).
   \end{itemize}
Further, considering the map $\Phi_m: (\C^*)^d \lto (\C^*)^d$,  $z\lmto z^m,$ we have the isomorphism of the following:
\begin{itemize}
\item [(a')] The $\mathbb{Q}$-algebra of $\Phi_m$-invariant tropical currents on $(\C^*)^d$ with the usual addition of currents and the wedge product.
\item [(b')] The $\mathbb{Q}$-algebra generated by the tropical fans in $\R^d$ with the natural addition  and stable intersection. 
\end{itemize}
\end{itheorem}
Works of \cite{Fulton--Sturmfels}, \cite{Kazarnovski} and \cite{Jensen-Yu} also imply the isomorphisms of $\mathbb{Q}$-algebra in (b') with the McMullen algebra of dimensional polytopes in $\mathbb{Q}^d$ as well as the Chow group of the compact space obtained as a direct limit of all complete smooth toric varieties; see Theorem \ref{thm:Q-alg} for further details. 
\vskip 2mm
We also address Problem~\ref{prob:cont} in a specific case involving the slicing of currents converging to a tropical current:
\begin{itheorem}[See Theorem \ref{thm:limit-closure}]\label{ithm:limit-closure}
Let $\mathcal{C} \subseteq \mathbb{R}^d$ be a tropical variety, and let $B = A_1 \cap \dots \cap A_k \subseteq \mathbb{R}^d$ be a complete intersection of rational hyperplanes $A_1, \dots, A_k$. Assume that $\mathcal{C}$ and $B$ intersect transversely. Let $\Sigma$ be a smooth, projective fan compatible with $\mathcal{C} + B$, and let $(\overline{\mathscr{S}}_n)$ be a sequence of positive closed currents on $X_{\Sigma}$. Suppose that:
\begin{itemize}
    \item[(a)] $\overline{\mathscr{S}}_n \lto \overline{\mathscr{T}}_{\mathcal{C}}$;
    \item[(b)] $\supp{\overline{\mathscr{S}}_n} \lto \supp{\overline{\mathscr{T}}_{\mathcal{C}}}$.
\end{itemize}

Then, the following limit holds in the smooth projective toric variety $X_{\Sigma}$:
\[
\lim_{n \to \infty} \left(\overline{\mathscr{S}}_n \wedge [\overline{T}^{B}]\right) = \overline{\mathscr{T}}_{\mathcal{C}} \wedge [\overline{T}^B].
\]
\end{itheorem}

The proof relies on a theorem of Berteloot and the second-named author \cite{Berteloot--Dinh}, which states that the limit of slices satisfies a certain semi-continuity property for plurisubharmonic functions. This allows us to use Fourier analysis to prove theorems about tropical currents. We will see that this result, combined with \emph{dynamical tropicalisation} \cite{Dyn-trop}, implies analogues of the results in:
\begin{itemize}
    \item \cite{Jonsson} and \cite{Maclagan-Sturmfels}, concerning the fundamental theorem of tropical geometry and tropicalisation with respect to a non-trivial valuation (see Theorem \ref{thm:dyn-trop-nont});
    \item \cite{BJSST} and \cite{Osserman-Payne}, concerning the commutativity of intersection and tropicalisation (see Theorem \ref{thm:proper-int}).
\end{itemize}

\vskip 2mm
  It is notable that the setting of dynamical tropicalisation allows for the tropicalisation of differential forms and currents, enabling the formulation of theorems that would otherwise be impossible in the classical setting of tropical geometry. For instance:
\begin{itheorem}[See Theorem \ref{thm:averaging}]
    Let $W$ and $V$ be two smooth algebraic subvarieties of $(\C^*)^d$, and for $0<\epsilon <1,$ let $U_{\epsilon}((S^1)^d)$ be an $\epsilon$-neighbourhood of  $(S^1)^d.$ Then,
    $$
    \frac{1}{m^{2n-(p+q)}} \int_{(t_1, t_2)\in   U_{\epsilon}((S^1)^{2n})}\Phi_m^*\big( [t_1 V] \wedge [t_2 W] \big)\, d\nu(t_1) \otimes d\nu(t_2) \longrightarrow \sT_{\Trop_0(V)} \wedge \sT_{\Trop_0(W)},
    $$
    as $m \to \infty$. Here, the $d\nu$ is the normalised Lebesgue measures on $U_{\epsilon}((S^1)^d)$. 
\end{itheorem}
To this end, we hope that the computational aspects of tropical geometry, together with our results, will provide a testing ground for complex dynamics and superpotential theory. Furthermore, our results highlight the need to generalise our main theorems to a broader dynamical setting, which we leave for future work.

\subsection*{Acknowledgements}
This project has received funding from the National University of Singapore and the Ministry of Education (MOE) of Singapore through grant A-8002488-00-00. FB is also thankful for the support of the London Mathematical Society’s Research in Pairs programme (ref. 42239), which helped cover part of his travel expenses to Singapore in 2023.

\section{Tools from Superpotential and Slicing Theory of Currents}

Let \((X, \omega)\) be a compact K\"ahler manifold of dimension \(d\). Assume that \(\mathcal{S}\) is either a positive or a negative current of bidegree \((q, q)\) on \(X\). The quantity 
\[
| \langle \mathscr{S}, \omega^{d-q} \rangle |
\]
is referred to as the \emph{total mass} of \(\mathscr{S}\). For \(0 \leq r \leq d\), we consider the de~Rham cohomology groups \(H^r(X, \mathbb{C}) = H^r(X, \mathbb{R}) \otimes_{\mathbb{R}} \mathbb{C}\) with coefficients in \(\mathbb{C}\). Recall that Hodge theory provides the following decomposition of the de~Rham cohomology group into Dolbeault cohomology groups:
\[
H^r(X, \mathbb{C}) \simeq \bigoplus_{p+q=r} H^{p,q}(X, \mathbb{C}).
\]
We denote by \(\mathscr{C}^q(X)\) the cone of positive closed bidegree \((q,q)\)-currents or bidimension \((d-q, d-q)\) in \(X\). We denote by \(\mathscr{D}^q(X) = \mathscr{D}_{d-q}(X)\) the \(\mathbb{R}\)-vector space spanned by \(\mathscr{C}^q(X)\), which is the space of closed real currents of bidegree \((q, q)\). Every current \(\mathscr{T} \in \mathscr{D}^q(X)\) has a cohomology class:
\[
\{\mathscr{T}\} \in H^{q,q}(X, \mathbb{R}) = H^{q,q}(X, \mathbb{C}) \cap H^{2q}(X, \mathbb{R}).
\]
We define \(\mathscr{D}^{q,0}(X) = \mathscr{D}_{d-q}^0(X)\) to be the subspace of \(\mathscr{D}^q(X)\), consisting of currents with vanishing cohomology. The \(*\)-topology on \(\mathscr{D}^q(X)\) is defined by the norm:
\[
\|\mathscr{S}\|_* := \min(\|\mathscr{S}^+\| + \|\mathscr{S}^-\|),
\]
where the minimum is taken over positive currents \(\mathscr{S}^+\) and \(\mathscr{S}^-\) in \(\mathscr{C}^q(X)\) that satisfy \(\mathscr{S} = \mathscr{S}^+ - \mathscr{S}^-\). We say that \(\mathscr{S}_n\) converges to \(\mathscr{S}\) in \(\mathscr{D}^q(X)\) if \(\mathscr{S}_n\) converges weakly to \(\mathscr{S}\) and moreover, \(\|\mathscr{S}_n\|_{*}\) is bounded by a constant independent of \(n\).

\vskip 2mm
Let $h:= \dim H^{q,q}(X, \R),$ and fix a set of smooth forms $\alpha = (\alpha_1, \dots , \alpha_{h})$  such that their cohomology classes $\{ \alpha \}= (\{\alpha_1 \}, \dots , \{\alpha_{h}\})$ form a basis for $H^{q,q}(X, \R).$ By Poincar\'e duality, there exists a set of smooth forms $\alpha^{\vee} = ( \alpha_1^{\vee} , \dots , \alpha_{h}^{\vee})$ such that their cohomology classes $\{ \alpha^{\vee}\}$ form the dual basis of $\{\alpha\},$ with respect to the cup-product. By adding $U_{\sS}$ to a suitable combination of $\alpha_i^{\vee},$ we can assume that $\langle U_{\sS} , \alpha_i \rangle =0,$ for all $i = 1, \dots, h.$ In this case, we say that $U_{\sS}$ is \emph{$\alpha$-normalised}. 
\begin{definition}\label{def:CSP}
Let $\sT \in \sD^{q}(X)$ and $\sS$ be a smooth form in $\sD^{d-q+1,0}(X).$ 
\begin{itemize}
    \item [(a)] The \emph{$\alpha$-normalised superpotential}  $\mathscr{U}_{\sT}$ of $\sT$ is given by the function 
\begin{align*}
    \sU_{\sT}  :\{ \sS \in \sD^{d-q+1,0}(X): \text{smooth} \} & \lto \R  \\ 
     \sS &\lmto \langle \sT , U_{\sS} \rangle,
\end{align*}
where $U_{\sS}$ is the $\alpha$-normalised potential of $\sS.$
    \item [(b)] We say $\sT$ has a \emph{continuous superpotential}, if $\sU_{\sT}$ can be extended to a function on $\sD^{d-q+1, 0}(X)$ which is continuous with respect to the $*$-topology. 
\end{itemize}

\end{definition}

In general, consider $\sT \in \sD^q(X)$ and $\sS \in \sD^r(X).$ Assume that $q+r \leq d$ and $\sT$ has a continuous superpotential. Let $\sU_{\sT}$ be the $\alpha$-normalised superpotential of $\sT.$ Let $\beta \in \text{Span}_{\R}\{\alpha\}$ such that $ \{\beta \} = \{\sT \}.$ For any compactly supported smooth form $\varphi$ of bidegree $(d-q-r, d-q-r),$ we define 
\begin{equation}\label{eq:sp-formula}
    \langle \sT \wedge \sS , \varphi \rangle := \sU_{\sT}(  \sS \wedge dd^c \varphi) + \langle \beta \wedge \sS , \varphi  \rangle. 
\end{equation}
Now assume that if $f: X \lto Y,$ is a biholomorphism  between smooth compact K\"ahler manifolds, then we have
$$
 f_* \sU_{\sR_1} = \sU_{f_* \sR_1}, \quad   f^* \sU_{\sR_2} = \sU_{f^* \sR_2},
$$
for $\sR_1 \in \sD^{q}(X)$ and $\sR_2 \in \sD^{q}(Y).$


\begin{definition}
    Let $(\sT_n)$ be a sequence of currents in $\sD^q(X)$ weakly converging to $\sT.$ Let $\sU_{\sT}$ and $\sU_{\sT_n}$ be their $\alpha$-normalised superpotentials. If $\sU_{\sT_n}$ converges to $\sU_{\sT}$ uniformly on any $*$-bounded sets of smooth forms in $\sD^{d-q+1,0}(X),$ then the convergence is called \emph{$\textrm{SP}$-uniform}.
\end{definition}
It is shown in \cite{Dinh-Sibony:Kahler}*{Proposition 3.2.8} that any current with continuous superpotentials can be SP-uniformly approximated by smooth forms. Moreover, currents with continuous superpotentials have other nice properties:

\begin{theorem}[\cite{Dinh-Nguyen-Vu}*{Theorem 1.1}]\label{thm:superpot-ineq}
  Suppose that $\sT$ and $\sT'$ are two positive currents in $\sD_q(X)$, such that $\sT \leq \sT',$ \textit{i.e.}, $\sT' -\sT$ is a positive current. Then, if $\sT'$ has a continuous superpotential, then so does $\sT$.
\end{theorem}

\begin{theorem}\label{thm:wedge-cont-pot}
 If $\sT_1$ and $\sT_2$ are two positive closed currents, and $\sT_1$ has a continuous superpotentials, then $\sT_1 \wedge \sT_2$ is well-defined. Moreover, if $\sT_2$ has also a continuous superpotential, then 
 \begin{itemize}
     \item [(a)] \cite{Dinh-Sibony:Kahler}*{Proposition 3.3.3} $\sT_1 \wedge \sT_2$ has a continuous superpotential; 
     \item [(b)] \cite{Dinh-Sibony:Kahler}*{Proposition 3.3.3} This wedge product is continuous with respect to the SP-uniform convergence. 
     \item [(c)] \cite{Dinh-Sibony:superpot}*{Theorem 4.2.4} $\supp{\sT_1 \wedge \sT_2} \subseteq \supp{\sT_1} \cap \supp{ \sT_2}.$
 \end{itemize}

\end{theorem}
\begin{theorem}[\cite{Dinh-Sibony:Kahler}*{Proposition 3.3.4}]\label{thm:DS-commut-asso}
    Assume that $\sT_1, \sT_2$ and $\sT_3$ are closed positive currents, and $\sT_1$ and $\sT_2$ have continuous superpotentials. Then, 
$$
\sT_1 \wedge \sT_2 = \sT_2 \wedge \sT_1 \quad \text{and} \quad (\sT_1 \wedge \sT_2) \wedge \sT_3 = \sT_1 \wedge (\sT_2 \wedge \sT_3). 
$$
\end{theorem}
\begin{proposition}\label{prop:slicing-CSP}
   Let $X$ be a compact K\"ahler manifold, $S_n \lto S$ be a convergent sequence in $\sD^q(X).$ If a current $\sT$ has a continuous superpotential, then 
    $$
    \sT \wedge \sS_n \lto \sT \wedge \sS. 
    $$
\end{proposition}
\begin{proof}
    The main result of \cite{Dinh-Sibony-Reg} implies any current $\sT \in \sD^p(X)$ can be weakly approximated by a difference of smooth closed positive of bidegree $(p,p)$-forms. The result then follows from the definition of continuity of superpotentials and Equation (\ref{eq:sp-formula}). 
\end{proof}

\begin{lemma}\label{lem:local-equality}
    Let $\sT,$ $\sT'$ be positive closed currents such that $\sT_{|_\Omega} = \sT'_{|_\Omega}$ in an open subset $\Omega \subseteq X,$ and both $\sT$ and $\sT'$ have continuous superpotentials. Then, for any $\sS \in \sD^{r}(X),$ 
    $$
     \big(\sT\wedge \sS\big) _{|_\Omega} =   \big(\sT' \wedge \sS \big) _{|_\Omega}.
    $$
\end{lemma}
\begin{proof}
In \cite{Dinh-Sibony:Kahler}, for any current $\sS$ with continuous superpotential, a family $\{\sT_{\theta}\}_{\theta \in \C^*}$ is constructed that $\sT_{\theta}$ converges SP-uniformly to $\sS$ as $|\theta| \to 0.$  Let $\epsilon :=|\theta|>0,$ be a small positive number, and $V\subseteq U$ be any open set such that $V_{\epsilon}$, the $\epsilon$-neighbourhood of $V$, is contained entirely in $\Omega.$ Therefore, by the hypothesis of the lemma, we can construct families of smooth forms $\sT_n$ and $\sT_n'$ converging $\textrm{SP}$-uniformly to $\sT$ and $\sT'$ respectively. Moreover,  
$${\sT_n}_{|_{V_{\epsilon}}} = {\sT'_n}_{|_{V_{\epsilon}}}.$$ 
Now, for any $(d-q-r, d-q-r)$ smooth form $\varphi$ with compact support on $\Omega,$ we can cover the support of $\varphi$ with an open set of the form $V_{\epsilon}$ and deduce
$$
     \big(\sT_n \wedge \sS \big)\wedge \varphi =   \big(\sT_n' \wedge \sS \big) \wedge \varphi.
$$
This, together with Theorem \ref{thm:wedge-cont-pot}(b) implies the assertion.

\end{proof}
\begin{theorem}\label{thm:CS-blowup}
  Let $q: \widehat{X} \lto X,$ be the blowing up of the compact K\"ahler manifold $X$ along a submanifold. Assume that $\sT \in \sD_p(\widehat{X})$ is such that the support of $\sT$ does not intersect the exceptional divisors of $\widehat{X}.$ If   the   current $\sT$ has a continuous superpotential then $q_* \sT$ has the same property. 
\end{theorem}
\begin{proof} Assume that $\sS\in \sD^{d-p+1}(X)$. We need to show that the superpotential of $q_* \sT$, which is a function defined on smooth forms in $\sD^{d-q+1,0}(X)$, can be extended to a continuous function on 
$\sD^{d-q+1,0}(X)$. Let $\alpha$ be a smooth closed $(p,p)$-form cohomologous to $q_*\sT$ and define $\beta:=q^*\alpha$. By hypothesis, we have $q^*q_*\sT=\sT$. It follows that $\sT$ is cohomologous to $\beta$. Fix a potential $U$ of $q_* \sT-\alpha$ which is smooth near the blowup locus. We have $dd^c  U= q_* \sT-\alpha$ and therefore 
$dd^c q^*U = \sT-\beta$. The smoothness of $U$ near the blowup locus implies that $q_*(q^*U)=U$.

Since $\alpha$ is smooth, it is enough to show that the superpotential $\sU$ of $q_* \sT-\alpha$ can be extended to a continuous function on $\sD^{d-q+1,0}(X)$. Since the last current has a vanishing cohomology class, its superpotential doesn't depend on the normalization.

\bigskip\noindent
{\bf Claim 1.}  Let $(\sS_n)$ be a bounded sequence of smooth forms in $\sD^{d-p+1}(X)$. Then the sequence $\sU(\sS_n)$ is bounded. 

\proof[Proof of Claim 1]
Since $\sS_n$ is smooth, we have
$$\sU(\sS_n)=\langle U,\sS_n\rangle = \langle q^*(U), q^*\sS_n \rangle = \widehat\sU(q^*\sS_n), $$
where $\widehat \sU$ denotes the superpotential of $\sT-\beta$.
Since the action of $q^*$ on cohomology is bounded and the mass of a positive closed current only depends on its cohomology class, we see that $q^*\sS_n$ is bounded in $\sD^{d-p+1,0}(\widehat X)$. Since $\sT$ has a continuous superpotential, we deduce that $\widehat\sU(q^*\sS_n)$ is bounded and hence $\sU(\sS_n)$ is bounded as claimed.
\endproof

\bigskip\noindent
{\bf Claim 2.}  Let $(\sS_n)$ be a bounded sequence of smooth forms in $\sD^{d-p+1}(X)$ converging to 0. Then $\sU(\sS_n)$ tends to 0.

\proof[Proof of Claim 2]
By extracting a subsequence, we can assume that $q^*\sS_n$ converges to some current $\sR$ supported by the exceptional divisor. 
By hypothesis, we have $q_*(\sR)=0$.
Using the computation in Claim 1, and the fact that $\widehat \sU$ is continuous, we get
$$\lim  \sU(\sS_n) = \widehat\sU (\sR).$$
Now, since $q^*U$ is smooth near the exceptional divisor which supports $\sR$, we deduce that 
$$\widehat\sU (\sR)=\langle q^*U,\sR\rangle = \langle U, q_*\sR\rangle =0.$$
This proves the claim.
\endproof

To finish the proof, let $\sS$ be any current in $\sD^{d-p+1,0}(X)$. Choose a bounded sequence $\sR_n$ of smooth forms in $\sD^{d-p+1,0}(X)$ converging to $\sS$. 
By Claim 1, extracting a subsequence allows to assume that $\sU(\sR_n)$ converges to some real number $l$. Consider now an arbitrary bounded sequence of 
$(\sS_n)$ of smooth forms in $\sD^{d-p+1,0}(X)$ converging to $\sS$. To show that $\sU$ extends to a continuous function at $\sS$, it is enough to check 
that $\sU(\sS_n)$ converges to $l$. This is a consequence of Claim 2 applied to the sequence $\sS_n-\sR_n$.
\end{proof}



\begin{theorem}\label{thm:tensor-conv}
  For two complex manifolds $X$ and $Y,$  consider two convergent sequences of currents $\sT_n \lto \sT$ in $\sD^q(X)$ and $\sS_n \lto \sS$ in $\sD^r(Y)$. We have that 
    $$
     \sT_{n} \otimes \sS_n \lto \sT \otimes \sS,
    $$
    weakly in $\sD^{q+r}(X \times Y).$
\end{theorem}
\begin{proof}[Sketch of the proof]
Let us denote by $(x,y)$ the coordinates on $X \times Y$. Using local coordinates and a partition of unity and Weierstrass theorem we can approximate any smooth forms on $X \times Y$ with forms with polynomial coefficients in $(x,y).$ The approximation is in $C^\infty$. As a result, the convergence, we only need test forms with monomial coefficients. Thus, the variables $x, y$ are separated, and the convergence of the tensor products becomes the convergence of each factor.\end{proof}
\subsection{Semi-continuity of slices} Let $f:X\lto Y$ be a dominant holomorphic map between complex manifolds, not necessarily compact, of dimension $d$ and $m$ respectively. Let $\sT$ be a positive closed current on $X$ of bi-dimension $(p,p)$ with $p\geq m$. Then a slice
$$\sT^f_y=\langle \sT|f|y\rangle$$
obtained by restricting $\sT$ to $f^{-1}(y)$ exists for almost every $y\in Y$; see \cite{DemaillyBook1}*{Page 171}. This is a positive closed current of bi-dimension $(p-m,p-m)$ on $X$ with the support contained $f^{-1}(y)$. If $\Omega$ is a smooth form of maximal bi-degree on $Y$ and $\alpha$ a smooth $(p-m,p-m)$-form with compact support in $X$, then we have
$$
\langle \sT,\alpha\wedge f^*(\Omega)\rangle = \int_{y\in Y} \langle \sT^f_y,\alpha\rangle \Omega(y).
$$
In general, for a fixed $f$, the equality $\langle \sT|f|y\rangle = \langle \sT'|f|y\rangle$ for almost every $y$, does not imply that $\sT=\sT'$. However, the following is true: let $f_1,\ldots,f_k$ be dominant holomorphic maps from $X$ to $Y_1,\ldots, Y_k$. Consider the vector space spanned by all the differential forms of type $\alpha\wedge f_i^*(\Omega_i)$ for some $\alpha$ as above and some smooth form $\Omega_i$ on $Y_i$ of maximal degree. Assume this space is equal to the space of all $(p,p)$-forms of compact support in $X$. Then if $\langle \sT|f_i|y_i\rangle = \langle \sT'|f_i|y\rangle$ for every $i$ and almost every $y\in Y_i$, we have $\sT=\sT'$. The proof is a consequence of the above discussion. 
\vskip 2mm
Now let $ U \subseteq \C^m $ and $V \subseteq \C^d$ be two bounded open sets. Assume that $\pi_1 : U \times V \lto U$ and $\pi_2: U\times V \lto V$ are the canonical projections. Consider two closed positive currents $\sT$ and $\sS$ on $U \times V$ of bi-dimension $(m,m)$ and $(d,d)$ respectively. We say that $\sT$ horizontal-like if  $\pi_2(\supp{\sT})$ is relatively compact in $V$. Similarly, if $\pi_1(\supp{\sS})$ is relatively compact in $U$, $\sS$ is called  vertical-like.

\begin{theorem}[\cite{Berteloot--Dinh}*{Lemma 3.7}]\label{thm:Berteloot-Dinh}
    Let $(\sT_n) \lto \sT$ be a convergent sequence of horizontal-like positive closed currents to a horizontal-like current $\sT$  in $U \times V.$  Let $a\in U$ and assume that the sequence of measures $(\langle \sT_n, \pi_1 , a \rangle)_n$ is also convergent. Then, 
    $$
     \lim_{n \to \infty } \langle \sT_n| \pi_1 | a \rangle (\phi)  \leq  \langle \sT| \pi_1 | a \rangle (\phi) 
    $$
    for every plurisubharmonic function $\phi$ on $\C^d.$
    
\end{theorem}
There is a simple version of the above theorem for supports which will be useful later. 
\begin{lemma}\label{lem:slice-supp}
Assume that the $\sT_n$'s, $\sS$, and $\sT$ are all closed positive currents, and that $\sT_n \wedge \sS$ and $\sT \wedge \sS$ are well-defined. Further, suppose that we have the convergent sequences
$$
\sT_n \lto \sT, \quad  \supp{\sT_n} \lto \supp{\sT},
$$
in the sense of currents and in the Hausdorff metric for supports. Then,
$$
\supp{ \lim (\sT_n \wedge \sS)} \subseteq  \supp{\sT} \cap \supp{\sS}.
$$
\end{lemma}

\begin{proof}
For a point $z$ outside the support of $\sT$, there exists a sufficiently small radius $\epsilon$ such that, for sufficiently large $n$, the current $\sT_n$ vanishes on the ball $B_{\epsilon}(z)$ centred at $z$. It follows that any limit of $\sT_n \wedge \sS$ vanishes on $B_{\epsilon}(z)$. So its support does not contain $z$. Moreover, its support does not contain any point outside $\supp{\sS}$.
\end{proof}

\section{Tropical Cycles, Tori, Tropical Currents}

In this section, we recall the definition of tropical cycles and note that with the natural addition of tropical cycles and their \emph{stable intersection}, the tropical cycles form a ring.

\subsection{Tropical varieties}
A linear subspace $H \subseteq \R^d$ is said to be \textit{rational} if there exists a subset of $\Z^d$ that spans $H$. A \emph{rational polyhedron} is the intersection of finitely many rational half-spaces defined by
\[
\{x \in \R^d : \langle m, x \rangle \geq c, \text{ for some } m \in \Z^d,~ c \in \R \}.
\]
A \emph{rational polyhedral complex} is a polyhedral complex consisting solely of rational polyhedra. The polyhedra in a polyhedral complex are also referred to as \textit{cells}. A \textit{fan} is a polyhedral complex whose cells are all cones. If every cone in a fan $\Sigma$ is contained in another fan $\Sigma'$, then $\Sigma$ is called a \emph{subfan} of $\Sigma$. The one-dimensional cones of a fan are often called \emph{rays}. Throughout this article, all fans and polyhedral complexes are assumed to be \emph{rational}.

\vskip 2mm
For a given polyhedron $\sigma,$ and a finitely generated abelian group $N$, we denote by 
\begin{eqnarray*}
    \text{aff}(\sigma)&:= &\text{affine span of $\sigma$}, \\
    H_{\sigma}&:=& \text{translation of $\text{aff}(\sigma)$ to the origin}, \\
    N_{\sigma} &:=& N \cap H_{\sigma},\\
    N(\sigma)&:=& N/N_{\sigma}.
\end{eqnarray*}

Consider $\tau,$ a codimension one face of a $p$-dimensional polyhedron $\sigma,$ and let $u_{\sigma / \tau }$ be the unique outward generator of the one-dimensional lattice $(\Z^d \cap H_{\sigma})/(\Z^d \cap H_{\tau}).$

\begin{definition}[Balancing Condition and Tropical Cycles]\label{BalancingCondition}
Let $\mathscr{C}$ be a $p$-dimensional polyhedral complex whose $p$-dimensional cones are equipped with integer weights $w_{\sigma}$. We say that $\cC$ satisfies the \emph{balancing condition} at $\tau$ if
\[
\sum_{\sigma\supset \tau} w(\sigma)~ u_{\sigma/\tau}=0, \quad \text{in } \Z^d /(\Z^d \cap H_{\tau}),
\]
where the sum is over all $p$-dimensional cells $\sigma$ in $\mathcal{C}$ containing $\tau$ as a face. A  \emph{tropical cycle} or a \emph{tropical variety} in $\R^d$ is a weighted complex with finitely many cells that satisfies the balancing condition at every cone of dimension $p-1$. We say a tropical cycle is a \emph{tropical fan} if its support is a fan.
\end{definition}

\subsection{Stable intersection and addition of tropical cycles}


Recall that, generally speaking, the star of a cone in a complex is the extension of the local $p$-dimensional fan surrounding it. More precisely:

\begin{definition} Given a polyhedral complex $\Sigma \subseteq \mathbb{R}^d$ and a cell $\tau \in \Sigma$, define the star of $\tau$ in $\Sigma,$ denoted by $\text{star}_{\Sigma}(\tau),$ is a fan in $\mathbb{R}^d$. The cones of $\text{star}_{\Sigma}(\tau)$ are the \emph{extensions} of cells $\sigma$ that include $\tau$ as a face. Here, by extension, we mean
    \[
    \bar{\sigma} = \{\lambda (x - y) : \lambda \geq 0,  x \in \sigma,\, y \in \tau\}.
    \]
\end{definition}

\begin{definition}[Stable Intersection]\label{def:stable-int}
\begin{itemize}
    \item [(a)] Let $\cC_1, \cC_2 \subseteq \R^d$ be two positively weighted tropical cycles of dimension $p$ and $q,$ intersecting transversely. That is, the top dimensional cells $\sigma_1 \in \cC_1$ and $\sigma_2 \in \cC_2$ intersect in dimension $p+q-d$ and in the relative interior of these cells.  Then the stable intersection of $\cC_1 \cdot \cC_2$ is the tropical cycles supported on finitely many cells $\cC_1 \cap \cC_2.$ In this case, the weight of a cell $\sigma_1 \cap \sigma_2$ if defined by
    \[
    w_{\cC_1 \cdot \cC_2} (\sigma_1 \cap \sigma_2) = w_{\sigma_1} w_{\sigma_2} [N: {N}_{\sigma_1} + {N}_{\sigma_2}],
    \]
    where $N = \Z^d$ here.  
    \item [(b)] When $\cC_1$ and $\cC_2$ do not intersect transversely, then  $\cC_1 \cdot \cC_2$ as a set is the Hausdorff limit of 
    $$\cC_1 \cap  (\epsilon b + \cC_2), \quad \text{as $\epsilon \to 0$},$$ 
    for a fixed generic $b \in \R^d,$ and the weights are the sum of all the tropical multiplicities of the cells in the transversal intersection $\cC_1 \cap (\epsilon b + \cC_2)$ which converge to the same $(p+q-d)$-dimensional cell in the Hausdorff metric. Equivalently, for top dimensional cones $\sigma_1 \in \cC_1$ and $\sigma_2 \in \cC_2$
     \[
    w_{\cC_1 \cdot \cC_2} (\sigma_1 \cap \sigma_2) = \sum_{\tau_1, \tau_2} w_{\tau_1} w_{\tau_2} [{N}: {N}_{\tau_1} + {N}_{\tau_2}],
    \]
 where the sum is taken over all $\tau_1 \in \text{star}_{\cC_1}(\sigma_1 \cap \sigma_2), \tau_2 \in \text{star}_{\cC_2}(\sigma_1 \cap \sigma_2)$ with $\tau_1 \cap (\epsilon b + \tau_2) \neq \varnothing,$ for some fixed generic vector $b \in \mathbb{R}^d.$

 \item [(c)]  When $p+q < d,$ then the stable intersection of $\cC_1$ and $\cC_2$ is the empty set.

\end{itemize}

\end{definition}

The following result is proved in tropical geometry; see \cite{Maclagan-Sturmfels}*{Lemmas 3.6.4 and 3.6.9}. We will revisit its proof later through the lens of superpotential theory.
\begin{theorem}\label{thm:stable-invar}
When $p+q \geq d,$ the stable intersection, defined above, yields a balanced polyhedral complex of dimension $p+q-d$.
\end{theorem}
We also need the following for turning the set of tropical cycles into a $\Z$-algebra. 
\begin{definition}[Addition of Tropical Cycles]\label{def:add-trop}
    For two $p$-dimensional tropical cycles $\cC_1, \cC_2$ in $\mathbb{R}^d,$ the addition $\cC_1 + \cC_2$ is the tropical cycle obtained by the common refinement of the support $|\cC_1| \cup |\cC_2|$ where the weights of a cone $\sigma$ in the refinement are determined by $w_{\cC_1 + \cC_2}(\sigma) = w_{\cC_1}(\sigma) + w_{\cC_2}(\sigma).$
\end{definition}

Let us end this section with an example of the stable intersection. 
\begin{example}\label{eg:stable-int}
Consider the tropical cycles in the Figure \ref{fig:stable-int}. We can easily check that both cycles with the given weights satisfy the balancing condition. The line $\{ x_2=0\}$ with weight $1$ can be considered a tropical cycle that does not properly intersect the blue tropical variety. We can therefore translate by the vector $(0, -\epsilon )$ it to obtain $x_2= -\epsilon,$ and compute the stable intersection by taking the limit as $\epsilon \to 0.$ Calculating the multiplicities and Hausdorff limit gives the stable intersection equal to the origin $(0,0)$ with multiplicity 2. Note that we can choose any $\epsilon b \in \R^2$ generic for this translation, and thanks to Theorem \ref{thm:stable-invar} we obtain the same result.
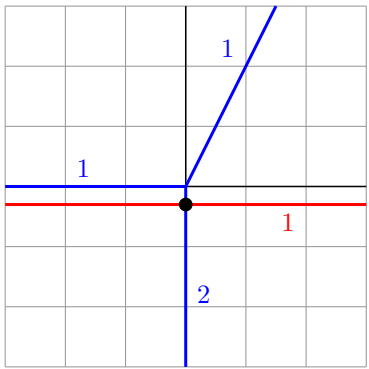
\begin{figure}
   
\begin{center}
  \begin{tikzpicture}[scale=0.8]
  \foreach \x in {-3,-2,...,3} {
    \draw[black!40] (\x,-3) -- (\x,3);
  }
  \foreach \y in {-3,-2,...,3} {
    \draw[black!40] (-3,\y) -- (3,\y);
  }

  \draw[line width=0.2mm,  black] (0,0) -- (3,0) node[below right] {$ $};
  \draw[line width=0.2mm,  black] (0,0) -- (0,3) node[above left] {$ $};

  \draw[line width=0.4mm,  blue]  (0,0) -- (0,-3);         
  \draw[line width=0.4mm,  blue]  (0,0) -- (-3,0);         
  \draw[line width=0.4mm,  blue]  (0,0) -- (1.5,3);        

  \node[blue] at (0.3,-1.8) {\small $2$};
  \node[blue] at (-1.7,0.3) {\small $1$};
  \node[blue] at (0.7,2.3) {\small $1$};


  \draw [line width=0.4mm,  red]  (-3,-0.3) -- (3,-0.3);      
  \node[red] at (1.7,-0.6) {\small $1$};
  \filldraw[black] (0,-0.3) circle (3pt);
\end{tikzpicture}
\end{center}

    \caption{An example of the stable intersection}
    \label{fig:stable-int}

\end{figure}
\end{example}

\section{Tropical Currents}
Closed currents associated with tropical varieties in the context of supercurrents first appeared in Lagerberg’s thesis (see \cite{Lagerberg}). The introduction of complex tropical currents as \emph{laminar currents} in \cite{Babaee}, outlined below, made it convenient to examine the extremality, cohomological, and dynamical properties of tropical currents; see \cite{BH, Adi-Baba, Dyn-trop}. A nice and simple relation between Lagerberg currents and complex tropical currents is discussed in \cite{Gil-Gubler-Jell-Kunnemann}.
\vskip 2mm
For a finitely generated abelian group $N$ of rank $d,$ we define:
\begin{eqnarray*}
T_N&:=&\text{the complex algebraic torus $\mathbb{C}^* \otimes_\mathbb{Z} N$,}\\
S_N&:=&\text{the compact real torus $S^1 \otimes_\mathbb{Z} N$,}\\
N_\mathbb{R}&:=&\text{the real vector space $\mathbb{R} \otimes_\mathbb{Z} N$.}
\end{eqnarray*}

Let \( \mathbb{C}^* \) denote the multiplicative group of nonzero complex numbers. As before, we define the following canonical maps on the complex torus \( (\mathbb{C}^*)^d \):

\begin{itemize}
    \item The \emph{logarithm map}:
    \[
    \mathrm{Log} : (\mathbb{C}^*)^d \longrightarrow \mathbb{R}^d, \qquad (z_1, \dots, z_d) \longmapsto (-\log |z_1|, \dots, -\log |z_d|).
    \]

    \item The \emph{argument map}:
    \[
    \mathrm{Arg} : (\mathbb{C}^*)^d \longrightarrow (S^1)^d, \qquad (z_1, \dots, z_d) \longmapsto \left( \frac{z_1}{|z_1|}, \dots, \frac{z_d}{|z_d|} \right).
    \]
\end{itemize}

Let \( H \subseteq \mathbb{R}^d \) be a rational linear subspace. We have the following exact sequence of abelian groups:
\[
\xymatrix{
0 \ar[r] & H \cap \mathbb{Z}^d \ar[r] & \mathbb{Z}^d \ar[r] & \mathbb{Z}^d / (H \cap \mathbb{Z}^d) \ar[r] & 0.
}
\]

This induces a corresponding exact sequence of real tori:
\[
\xymatrix{
0 \ar[r] & S_{H \cap \mathbb{Z}^d} \ar[r] & (S^1)^d \cong S^1 \otimes_{\mathbb{Z}} \mathbb{Z}^d \ar[r] & S_{\mathbb{Z}^d / (H \cap \mathbb{Z}^d)} \ar[r] & 0.
}
\]

Define the map
\[
\pi_H : \mathrm{Log}^{-1}(H) \xrightarrow{\mathrm{Arg}} (S^1)^d \longrightarrow S_{\mathbb{Z}^d / (H \cap \mathbb{Z}^d)}.
\]

\vspace{2mm}

Similarly, there is an exact sequence of complex tori:
\[
\xymatrix{
0 \ar[r] & T_{H \cap \mathbb{Z}^d} \ar[r] & (\mathbb{C}^*)^d \cong \mathbb{C}^* \otimes_{\mathbb{Z}} \mathbb{Z}^d \ar[r] & T_{\mathbb{Z}^d / (H \cap \mathbb{Z}^d)} \ar[r] & 0.
}
\]

Define the projection map:
\[
\Pi_H : (\mathbb{C}^*)^d \cong \mathbb{C}^* \otimes \left( (H \cap \mathbb{Z}^d) \oplus \mathbb{Z}^d / (H \cap \mathbb{Z}^d) \right) \longrightarrow T_{\mathbb{Z}^d / (H \cap \mathbb{Z}^d)}.
\]

Then we have the identity:
\[
\ker(\Pi_H) = \ker(\pi_H) = T_{H \cap \mathbb{Z}^d} \subseteq (\mathbb{C}^*)^d.
\]

As a result,  when $H$ is of dimension $p,$ the set $\text{Log}^{-1}(H)$ is naturally foliated by the  $\pi_H^{-1}(x) = T_{H \cap \mathbb{Z}^d}\cdot x \, \simeq (\C^*)^p$ for $x \in S_{\mathbb{Z}^d/(H \cap \mathbb{Z}^d)}.$ For a lattice basis $u_1, \dots , u_p$ of $H \cap \Z^d,$ the tori $T_{H \cap \Z^d}\cdot x$ can be parametrised by the monomial map
$$
 (\C^*)^p \lto (\C^*)^d, \quad z \lmto x\cdot  z^{[u_1, \dots , u_p]^{t}}
$$
where $U=[u_1, \dots , u_p]$ is the matrix with column vectors $u_1, \dots , u_p,$ and $z^{U^{t}}$ denotes that $z \in (\C^*)^p$ is taken to have the exponents with rows of the matrix $U.$ Accordingly, one can easily check that 
$$
T_{H\cap \Z^d}\cdot x =\{z\in (\C^*)^d:  z^{m_i} = x^{m_i}, \, i=1, \dots , d-p  \} .
$$
for any choice of a $\Z$-basis $\{m_1, \dots , m_{d-p} \}$ of  $ \Z^d / (H \cap \Z^d).$ 
\begin{definition}\label{def:T-H}
Let $H$ be a rational subspace of dimension $p,$ and $\mu$ be the Haar measure of mass $1$ on $S_{\mathbb{Z}^d/(H \cap \mathbb{Z}^d)}$. We define a $(p,p)$-dimensional closed current $\mathscr{T}_H$ on $(\mathbb{C}^*)^d$ by
\[
\mathscr{T}_H:=\int_{x \in S_{\mathbb{Z}^d/(H \cap \mathbb{Z}^d)}} \big[\pi_H^{-1}(x)\big] \ d\mu(x).
\]
\end{definition}
When $A$ is a rational affine subspace of $\mathbb{R}^d$ parallel to the linear subspace $H= A - {a}$ for $a \in A$, we define $\sT_{A}$ by translation of $\sT_H.$ Namely, we define the submersion $\pi_{A}$ as the composition
\[
\pi_{A}:\xymatrix{\text{Log}^{-1}(A) \ar[r]^{e^{a} } & \text{Log}^{-1}(H) \ar[r]^{\pi_{H} }&S_{\mathbb{Z}^d/(H \cap \mathbb{Z}^d)}.}
\]
We set $T^A:= \pi_A^{-1}(1)= \ker \pi_A =e^{-a} T_{H \cap \Z^d},$ which is a fibre of $\sT_{A}.$

\begin{definition}
Let $\mathcal{C}$ be a weighted polyhedral complex of dimension $p$. The complex tropical current (hereafter tropical current) $\mathscr{T}_{\mathcal{C}}$ associated to ${\mathcal{C}}$ is given by
$$
\mathscr{T}_{\mathcal{C}}= \sum_{\sigma} w_{\sigma} ~ \mathbbm{1}_{\Log^{-1}(\sigma)}\mathscr{T}_{\text{aff}(\sigma)} ,
$$
where the sum runs over all $p$-dimensional cells $\sigma$ of $\cC.$
\end{definition}


\begin{theorem}[\cite{Babaee}]\label{thm:closed-balanced}
A weighted complex \(\mathcal{C}\) is balanced if and only if \(\mathscr{T}_{\mathcal{C}}\) is closed.
\end{theorem}

\begin{theorem}[\cite{Babaee}]\label{thm:hyper-ddc}
Any tropical current \(\sT_{\cC} \in \sD_{d-1}((\C^*)^d)\) is of the form \(dd^c [\fq \circ \Log]\), where \(\fq : \R^d \to \R\) is a tropical Laurent polynomial; that is,
\[
\fq(x) = \max_{\alpha \in A} \{ c_{\alpha} + \langle \alpha, x \rangle \},
\]
for some finite subset \(A \subseteq \Z^d\) and coefficients \(c_{\alpha} \in \R\).
\end{theorem}

\begin{remark}\label{rem:trop-hyper-def}
Note that the support of \(dd^c [\fq \circ \Log]\) is given by \(\Log^{-1}(\Trop(\fq))\), where \(\Trop(\fq)\) is the set of points \(x \in \R^d\) where \(\fq\) is not smooth. This set can be balanced with natural weights, which coincide with the weights of the closed current \(dd^c [\fq \circ \Log]\), and is called the tropical variety associated to \(\fq\). We can then write
\[
\sT_{\Trop(\fq)} = dd^c [\fq \circ \Log].
\]
For example, the reader can verify that in Example~\ref{eg:stable-int}, the tropical curves in the figure correspond to the singular loci of \(\max\{0, x_2\}\) and \(\max\{x_2, 2x_1, 0\}\). The weights are essentially the integer side lengths of the corresponding Newton polytopes. See \cite{Maclagan-Sturmfels} for details.
\end{remark}

\begin{proposition}[\cite{Dyn-trop}*{Proposition 4.6}]
Assume that \(\sT \in \cD_{p}((\C^*)^d)\) is a closed, positive, \((S^1)^d\)-invariant current whose support is given by \(\Log^{-1}(|\cC|)\), for a polyhedral complex \(\cC \subseteq \R^d\) of pure dimension \(p\). Then \(\sT\) is a tropical current. In particulare, there are natural induced weights on $\cC$ to make it balanced.
\end{proposition}

    

\section{Continuity of Superpotentials}

Let $\fq:\R^d \lto \R,$ be a tropical polynomial function, and $\Log: (\C^*)^d \lto \R^d,$ as before. The current $dd^c [\fq \circ \Log] \in \sD_{d-1} ((\C^*)^d)$ has a bounded potential on any relatively compact open set, and by Bedford--Taylor theory, for any positive closed current $\sT \in \sD_{p} (\C^*)^d,$  the product 
$$
dd^c [ \fq \circ \Log] \wedge \sT = dd^c ([\fq\circ \Log] \, \sT),
$$
is well-defined. See \cite{DemaillyBook1}*{Section III.3}. In higher codimensions though, to prove that any two tropical currents have a well-defined wedge product, we utilise the superpotential theory \cite{Dinh-Sibony:superpot} on a compact K\"ahler manifold, and as a result, we extend the tropical currents to smooth compact toric varieties. 

\subsection{Tropical Currents on Toric Varieties}
In a toric variety $X_{\Sigma},$ for a cone $\sigma\in \Sigma,$ we denote by $\cO_{\sigma},$ the toric orbit associated with $\sigma.$ We have 
$$
X_{\Sigma} = \bigcup_{\sigma \in \Sigma} \cO_{\sigma}.
$$
We also set $D_{\sigma}$ to be the closure of $\cO_{\sigma}$ in the $X_{\Sigma},$ and $\Sigma(p)$  the $p$-dimensional skeleton of $\Sigma$, that is, the union of $p$-dimensional cells of $\Sigma$. Fibres of tropical currents are algebraic varieties with finite degrees and can be extended by zero to any toric variety, in consequence, any tropical current can be extended by zero to toric varieties. Moreover, with the following compatibility condition, we can ask for the extension of the fibres to intersect the toric invariant divisors transversally. 
\begin{definition}
\begin{itemize}
    \item [(a)] For a polyhedron $\sigma,$ its \emph{recession cone} is the convex polyhedral cone 
    $$
 \textrm{rec}(\sigma) = \{ b\in \R^d: \sigma+b \subseteq \sigma\} \subseteq H_{\sigma}.
    $$
       \item [(b)] Let $\cC$ be a $p$-dimensional balanced weighted complex in $\R^d,$ and $\Sigma$ a $p$-dimensional fan. We say that $\cC$ is \emph{compatible} with $\Sigma$, if $\textrm{rec}(\sigma) \in \Sigma$ for all $\sigma \in \cC.$
       
       \item [(c)] We say the tropical current $\sT_{\cC}$ is \emph{compatible} with $X_{\Sigma},$ if all the closures of the fibers $\pi_{\text{aff}(\sigma)}^{-1}(x)$ in $X_{\Sigma}$ of $\sT_{\cC}$ intersect the torus invariant divisors of $X_{\Sigma}$ transversely. 
\end{itemize}
\end{definition}

\begin{theorem}[\cite{BH}*{Lemma 4.10}]\label{thm:transverse-fiber}
  Let $\cC$ be a $p$-dimensional tropical cycle and $\Sigma$ be a fan. Assume that $\sigma \in \cC$ is a $p$-dimensional polyhedron and $\rho \in \Sigma$ is a one-dimensional cone. Then 
  \begin{itemize}
      \item [(a)] The intersection $D_{\rho} \cap \overline{\pi^{-1}_{\textrm{aff}(\sigma)}(x)}$ is non-empty and transverse if and only if $\rho \in \textrm{rec}(\sigma).$ Here $\overline{\pi^{-1}_{\textrm{aff}(\sigma)}(x)}$ corresponds the closure of a fiber of $\sT_{\text{aff}(\sigma)}$ in the toric variety $X_{\Sigma}.$ 
  
      \item [(b)] In particular, $\cC$ is compatible with $\Sigma$ if and only if $\sT_{\cC}$ is compatible with $X_{\Sigma}.$
       \end{itemize}     
\end{theorem}

For a tropical current $\sT_{\cC} \in \sD_{p}((\C^*)^d),$ and given a toric variety $X_\Sigma$ we denote its extension by zero $\overline{\sT}_{\cC} \in \sD_{p}(X_{\Sigma}).$
\begin{proposition}
For every tropical variety \(\cC\), there exists a smooth projective toric fan \(\Sigma\) compatible with a subdivision of \(\cC\).
\end{proposition}

\begin{proof}
By \cite{Gil-Sombra}, for \(\cC\) there exists a refinement \(\cC'\), and a complete fan \(\Sigma_1 \subseteq \R^d\) such that \(\cC'\) is compatible with \(\Sigma_1\). Applying the toric Chow lemma \cite{Cox-Little-Schenck}*{Theorem~6.1.18} and the toric resolution of singularities \cite{Cox-Little-Schenck}*{Theorem~11.1.9}, we can find a fan \(\Sigma\) which is a refinement of \(\Sigma_1\), and which defines a smooth projective variety \(X_{\Sigma}\). The tropical variety \(\cC''\), which is the refinement of \(\cC'\) induced by \(\Sigma\), satisfies the statement.
\end{proof}

\begin{remark}
    When $\cC'$ is a refinement of a tropical variety $\cC,$ then $\cC'$ is a tropical variety with natural induced weights. It is also easy to check that we have the equality of currents $\sT_{\cC} = \sT_{\cC'}$ in $(\C^*)^d;$ see \cite{BH}*{Section 2.6}. 

\end{remark}

\begin{lemma} \label{lem:ext-hyper-toric}
Let \( \fq: \mathbb{R}^d \to \mathbb{R} \) be a tropical Laurent polynomial, and let \( X_{\Sigma} \) be a smooth projective toric variety compatible with a subdivision of \( \Trop(\fq) \). Let \( \rho \in \Sigma(1) \). Assume that \( \zeta_0 \in D_{\rho} \cap \operatorname{supp} \left( \overline{dd^c[\fq \circ \Log]} \right) \), and let \( \Omega \) be a sufficiently small neighborhood of \( \zeta_0 \). Then the function \( \fq \circ \Log \in \mathrm{PSH}(\Omega \setminus D_{\rho}) \cap \mathscr{C}^0(\Omega \setminus D_{\rho}) \) can be extended to a function 
\[
u : \Omega \to \mathbb{R} \cup \{ +\infty \},
\]
such that:
\begin{itemize}
    \item[(a)] On \( \Omega \), we have \( u = g + \kappa \log |f| \), where \( g \) is a continuous function, \( f \) is a local defining equation for \( D_{\rho} \), and \( \kappa \) is a negative integer.  
    \item[(b)] On \( \Omega \), the equality \( dd^c u = \overline{\mathscr{T}}_{\Trop(\fq)} + \kappa [D_{\rho}] \) holds.
    \item[(c)] On \( \Omega \), we have \( \overline{\mathscr{T}}_{\mathcal{C}} = dd^c g \). In particular, \( \overline{\mathscr{T}}_{\Trop(\fq)} \) admits a continuous superpotential.
\end{itemize}
\end{lemma}

\begin{proof}
Assume that \(\fq = \max_{\alpha \in A} \{ c_{\alpha} + \langle \alpha, x \rangle \}\). Recall that
\[
\Log = (-\log |\cdot|, \dots, -\log |\cdot|).
\]
We write
\[
\fq \circ \Log = \log \max_{\alpha} \{ |e^{c_{\alpha}} z^{-\alpha}| \}.
\]
Assume that near \(\zeta_0\), the function \(\fq \circ \Log\) is given by
\[
\max \left\{ |e^{c_{\beta}} z^{-\beta}|,\, |e^{c_{\gamma}} z^{-\gamma}| \right\}.
\]
This implies that on \(\Log(\Omega \setminus D_{\rho})\), the function \(\fq\) is given by
\(\max \left\{ c_{\beta} + \langle \beta, x \rangle,\, c_{\gamma} + \langle \gamma, x \rangle \right\}.\)
For \(\fq = \max_{\alpha \in A} \{ c_{\alpha} + \langle \alpha, x \rangle \}\), we define:
\[
\mathrm{rec}(\fq) = \max_{\alpha \in A} \{ \langle \alpha, x \rangle \}.
\]
It is not hard to check that
\[
\mathrm{rec} \left( \Trop(\fq) \right) = \Trop \left( \mathrm{rec}(\fq) \right);
\]
see \cite{Maclagan-Sturmfels}*{Page 132}.

\vskip 2mm

We now show that by extending each \(z^{-\alpha}\) as a rational function to \(X_{\Sigma}\), the compatibility condition implies that \(\fq \circ \Log\) extends to \(X_{\Sigma}\). By \cite{Cox-Little-Schenck}*{Proposition~4.1.2}, the divisor of the extension of a character \(z^{\alpha}\) in \(X_{\Sigma}\) is given by:
\begin{equation}\label{eq:div}
\mathrm{Div}(z^{\alpha}) = \sum_{\rho \in \Sigma(1)} \langle \alpha, n_{\rho} \rangle D_{\rho},
\end{equation}
where \(n_{\rho}\) is the minimal generator of \(\rho\).
By assumption,
\[
D_{\rho} \cap \operatorname{supp} \left( \overline{dd^c[\mathrm{rec}(\fq) \circ \Log]} \right) \neq \varnothing.
\]
Theorem~\ref{thm:transverse-fiber} implies that
\[
n_{\rho} \in \mathrm{rec}(V_{\trop(\fq)}).
\]
Moreover, if  
\[
\zeta_1 \in D_{\rho} \cap \operatorname{supp}\left( \overline{dd^c[\mathrm{rec}(\fq) \circ \Log]} \right),
\]
then in a small neighborhood of \( \Log(\zeta_1) \), the function \( \mathrm{rec}(\fq)(x) \) takes the form
\[
\mathrm{rec}(\fq)(x) = \max\{ \langle \beta, x \rangle, \langle \gamma, x \rangle \}.
\]
By definition,
\[
n_{\rho} \in \mathrm{rec}(\Trop(\fq)) \quad \text{if and only if} \quad \kappa := \langle \beta, n_{\rho} \rangle = \langle \gamma, n_{\rho} \rangle.
\]
This, together with Equation~\eqref{eq:div}, implies that the extensions of \( z^{-\beta} \) and \( z^{-\gamma} \) as rational functions to \( X_{\Sigma} \) have the same vanishing order along \( D_{\rho} \), and we may write
\[
z^{-\beta} = f^{\kappa} \cdot \frac{g_1}{h_1}, \qquad z^{-\gamma} = f^{\kappa} \cdot \frac{g_2}{h_2}.
\]
Now, observe that on \( \Omega \setminus D_{\rho} \),
\[
\fq \circ \Log = \max \left\{ \log \left| e^{c_{\beta}} z^{-\beta} \right|, \log \left| e^{c_{\gamma}} z^{-\gamma} \right| \right\} = \kappa \log |f| + \max \left\{ \left| e^{c_{\beta}} \frac{g_1}{h_1} \right|, \left| e^{c_{\gamma}} \frac{g_2}{h_2} \right| \right\}.
\]
We must have \( \kappa < 0 \); otherwise, \( \fq \circ \Log = -\infty \) in \( \Omega \setminus D_{\rho} \). Consequently, the function \( \fq \circ \Log : \Omega \setminus D_{\rho} \to \mathbb{R} \) can be extended to
\[
u := \kappa \log|f| + \max\left\{ \left| e^{c_\beta} \frac{g_1}{h_1} \right|, \left| e^{c_\gamma} \frac{g_2}{h_2} \right| \right\}
\]
in \( \Omega \). Setting
\[
g = \max\left\{ \left| e^{-c_\beta} \frac{g_1}{h_1} \right|, \left| e^{-c_\gamma} \frac{g_2}{h_2} \right| \right\}
\]
implies (a). We have
\[
dd^c [\fq \circ \Log]_{\mid_{\Omega \setminus D_{\rho}}}
= \left( dd^c \log |f|^\kappa ~~ dd^c \log|g| \right)_{\mid_{\Omega \setminus D_{\rho}}}
= dd^c \log|g|_{\mid_{\Omega \setminus D_{\rho}}},
\]
since \( dd^c \log |f|^{\kappa} \) is holomorphic in \( \Omega \setminus D_{\rho} \). As a result of compatibility with \( X_{\Sigma} \), the current \( \overline{dd^c [\fq \circ \Log]} \) does not charge any mass in \( D_{\rho} \), and we obtain
\[
\overline{dd^c [\fq \circ \Log]} = dd^c \log|g|.
\]
This, together with Theorem~\ref{thm:hyper-ddc}, implies (c) and (b).

\end{proof}
\begin{lemma}\label{lem:affine-hyper}
Assume that $\sigma$ is $p$-dimensional and  $\text{aff}(\sigma)= H_1 \cap \dots \cap H_{d-p},$ is given as the transversal intersection   hyperplanes $H_i \subseteq \R^d.$
If $\Sigma$ is a smooth projective fan compatible with $\bigcup_i H_i,$ then 
\[
\overline{\sT}_{\mathrm{aff}(\sigma)} 
\leq 
\big(~ \overline{\sT_{H_1} \wedge \dots \wedge \sT_{H_{d-p}}}~\big)
=
\overline{\sT}_{H_1} \wedge \dots \wedge \overline{\sT}_{H_{d-p}}.
\]
\end{lemma}
\begin{proof}
By the definition of tropical currents, we have the inequality 
$${\sT}_{\text{aff}(\sigma)} \leq {\sT_{H_1} \wedge \dots \wedge \sT}_{H_{d-p}},$$
as currents in $(\C^*)^d,$ since the right-hand side might have multiplicities but the currents have the same support. Now, the wedge products in $X_{\Sigma}$  are well-defined by Lemma \ref{lem:ext-hyper-toric} and Theorem \ref{thm:wedge-cont-pot}. As both currents on both sides of the equation coincide on $(\C^*)^d$, the support of the current on the right-hand side contains the closure of the support of $\sT_{\aff(\sigma)}$ in $X_{\Sigma}.$ For the equality, note that compatibility with $\Sigma,$ implies that $\overline{\sT}_{H_i}$ has a zero mass in $X_{\Sigma}\setminus T_N.$
\end{proof}
\begin{theorem}\label{thm:trop-curr-CSP}
Let $\cC$ be a positively weighted tropical cycle of dimension $p$ compatible with a smooth, projective fan $\Sigma$, then $\overline{\sT}_{\cC}$ has a continuous superpotential in $X_{\Sigma}.$   
\end{theorem}
We need the following definition.
\begin{definition}\label{def:affine-ext}
  We define the affine extension $p$-dimensional a tropical cycle $\cC,$ by as the addition of tropical cycles  $$\widehat{\cC}:= \sum_{\sigma \in \cC} w_\sigma \text{aff}(\sigma).$$
\end{definition}
It is clear that if $\cC$ is a positively weighted tropical cycle, then $\sT_{\widehat{\cC}} - \sT_{\cC} \geq 0.$

\begin{proof}[Proof of \ref{thm:trop-curr-CSP} ]
  Let $ \widehat{\cC}$ be the affine extension of $\cC,$ and $\widehat{\Sigma}$ be a smooth projective fan which is a refinement of $\Sigma$ and compatible with $\widehat{\cC}.$ 
  By the preceding lemma and repeated application of Theorem~\ref{thm:wedge-cont-pot} for any $\sigma \in \cC,$ $\overline{\sT}_{\text{aff}(\sigma)}$ has a bounded superpotential, which implies this property for $\overline{\sT}_{\widehat{\cC}}.$ Now, since $\sT_{\widehat{\cC}} - \sT_{{\cC}}$ is a positive closed tropical current in $(\C^*)^d,$ $$ \overline{\sT_{\widehat{\cC}} - \sT_{\cC}} = \overline{\sT}_{\widehat{\cC}} - \overline{\sT}_{\cC}\geq 0
$$ in $X_{\widehat{\Sigma}}.$ Continuity of the superpotential of $\overline{\sT}_{\cC}$ in $X_{\widehat{\Sigma}}$ follows from Theorem \ref{thm:superpot-ineq}. 
\vskip 2mm
We now show that $\overline{\sT}_{\cC}$ has also a continuous superpotential on $X_{\Sigma}$ as well. We consider the proper map $f: X_{\widehat{\Sigma}} \lto X_{\Sigma},$ which can be understood as a composition of multiple blow-ups along toric points with exceptional divisors $D_{\rho}$ for any ray $\rho \in \widehat{\Sigma}\setminus \Sigma.$ These divisors satisfy $D_{\rho} \cap \supp{ \overline{\sT}_{\cC}} = \varnothing.$ We conclude the proof by Theorem \ref{thm:CS-blowup}.
\end{proof}

\begin{proposition}
   Let $\cC_1$ and $\cC_2$ be two positively weighted tropical cycles. If $X_{\Sigma}$ is compatible  $\cC_1 + \cC_2,$ then
   $$
    \overline{\sT_{\cC_1} \wedge \sT}_{\cC_2} = \overline{\sT}_{\cC_1} \wedge \overline{\sT}_{\cC_2}. 
   $$
   
\end{proposition}
\begin{proof}
The proof is clear since both $\overline{\sT}_{\cC_1}$ and $\overline{\sT}_{\cC_2}$ have continuous superpotentials with no mass on the boundary divisors $X_{\Sigma} \setminus T_N. $ 
\end{proof}

\begin{proposition}\label{prop:product-def-Cn}
    For any two positive tropical currents $\cC_1$ and $\cC_2$, the intersection product 
    $$
     \sT_{\cC_1} \wedge \sT_{\cC_2} :=  \overline{\sT}_{\cC_1} \wedge \overline{\sT}_{\cC_2}{|_{(\C^*)^d}},
    $$
    does not depend on the choice of a smooth projective toric variety of the fan $\Sigma$ compatible with $\cC_1 + \cC_2$, where $(\C^*)^d$ is identified with $T_N \subseteq X_{\Sigma}.$ Moreover, this product coincides with the definition of wedge products with bi-degree $(1,1)$ tropical currents in Bedford--Taylor Theory in $(\C^*)^d.$ 
\end{proposition}
\begin{proof}
This is a consequence of Lemma \ref{lem:local-equality}, and the fact that intersection product with a bidegree $(1,1)$ current in superpotential theory, in an open set of compact K\"ahler manifold, coincides with the Bedford--Taylor theory.
\end{proof}

 \begin{proposition}\label{prop:assoc-commut}
 The wedge product of positive tropical currents is associative and commutative.
\end{proposition}
\begin{proof}
  This is the application of Theorem \ref{thm:trop-curr-CSP} and Theorem \ref{thm:DS-commut-asso}.    
\end{proof}
\begin{remark}
For two currents $\sT \in \sD_p((\mathbb{C}^*)^d)$ and $\sS \in \sD_q((\mathbb{C}^*)^d)$, we define
\[
\sT \wedge \sS := (\sT \otimes \sS) \wedge \Delta,
\]
where $\Delta$ is the diagonal, hence a complete intersection in $(\mathbb{C}^*)^d \times (\mathbb{C}^*)^d$. Consequently, this intersection can be defined by repeatedly applying Bedford--Taylor theory to intersect with $\Delta$. In this context, the notions from Bedford--Taylor theory and slicing theory coincide.

To extend this definition to a toric variety $X$, we impose mild conditions on the supports to ensure that the current
\[
\overline{\sT} \wedge \overline{\sS} - \overline{\sT \wedge \sS} \in \sD_{p+q-n}(X)
\]
has support of Cauchy--Riemann dimension less than $(p+q-n)-1$, and thus vanishes; compare with \cite{BEGZ}. One can check that when $\sT$ and $\sS$ are integration currents on algebraic varieties, the multiplicities induced from this intersection coincide exactly with algebraic multiplicities.
\end{remark}

\subsection{Proof of $\sT_{\cC_1}\wedge \sT_{\cC_2} = \sT_{{\cC_1} \cdot {\cC_2}}$ }
We prove our main intersection theorems here, and it is not hard to visit the proof of Theorem \ref{thm:stable-invar} using tools from superpotential theory. 
\begin{theorem}\label{thm:main-int}
   For two positively weighted tropical varieties $\cC$ and $\cC'$ of dimension $p$ and $q,$ respectively, we have
    $$
     \sT_{\cC_1}\wedge \sT_{\cC_2} = \sT_{{\cC_1} \cdot {\cC_2}}
    $$
    where the $\cC \cdot \cC'$ is the stable intersection $\cC_1$ and $\cC_2,$ defined in Definition \ref{def:stable-int}. Moreover, the $\cC_1\cdot \cC_2$ is a balanced polyhedral complex of dimension $p+q -d. $
\end{theorem}

\begin{proposition}[\cite{Katz-toolkit}*{Propositions 6.1}]\label{prop:int-numb-tori}
    Let $H_1, H_2\subseteq \R^d$ be two rational planes of dimension $p$ and $q$ with $p+q=d$ that intersect transversely. Then, the complex tori
    $T_{H_1 \cap \Z^d}$ and $T_{H_2 \cap \Z^d}$ intersect at $[N: N_{H_1} + N_{H_2}]$ distinct points. 
\end{proposition}

\begin{proof}[Proof of Theorem \ref{thm:main-int}]
Note that $\sT_{\cC_1}\wedge \sT_{\cC_2}$ is well-defined by Proposition \ref{prop:product-def-Cn}. Assume that $\cC_1$ and $\cC_2$ are two tropical cycles of dimension $p$ and $q$, respectively. Note that when $p+q < d,$ both sides of the equality are zero. Therefore, we assume that $p+q \geq d.$ We proceed with the following steps. Firstly, we show that in the transversal case:
\begin{itemize}
    \item [(a)] $\supp{\sT_{\cC_1} \wedge \sT_{\cC_2}} = \Log^{-1}(\cC_1 \cdot \cC_2) = \Log^{-1}(\cC_1 \cap \cC_2)$;
    \item [(b)] When $p+q= d,$ $\sT_{\cC_1}\wedge \sT_{\cC_2} = \sT_{\cC_1 \cdot \cC_2}$; 
    \item [(c)] When $p+q >d,$ $\sT_{\cC_1}\wedge \sT_{\cC_2} = \sT_{\cC_1 \cdot \cC_2}$. 
\end{itemize}
Secondly, in general:
\begin{itemize}
    \item [(d)] $\Log(\supp{\sT_{\cC_1} \wedge \sT_{\cC_2}}) = \cC_1 \cdot \cC_2$ when $p+q=d$;
    \item [(e)]  $\Log(\supp{\sT_{\cC_1} \wedge \sT_{\cC_2}}) = \cC_1 \cdot \cC_2$ when $p+q >n$.
\end{itemize}
To see $(a),$ note the by Theorem \ref{thm:wedge-cont-pot}(c), $\supp{\sT_{\cC_1} \wedge \sT_{\cC_2}} \subseteq \Log^{-1}(\cC_1 \cdot \cC_2) = \Log^{-1}(\cC_1 \cap \cC_2).$ Moreover, on $\cC_1 \cap \cC_2,$ the fibres of $\sT_{\cC_1}$ and $\sT_{\cC_2}$ have a non-zero intersection. 
\vskip 2mm
\noindent
To prove (b), let $a \in \cC_1 \cap \cC_2$ be an isolated point of intersection. We can choose a small ball $B_{\epsilon}(a) \subseteq \R^d$ such that $a$ is an isolated point of intersection $\sigma_1 \cap \sigma_2 \cap B,$ where $\sigma_1$ and $\sigma_2$ are cells of dimension $p$ and $q$ in $\cC_1$ and $\cC_2$ respectively. For any rational polyhedron $\sigma,$ let
\begin{align*}
    N_{\sigma}&:= N \cap \aff(\sigma),\\
    S^1(\sigma)&:= S^1 \otimes_Z \big(\Z^d /(\Z^d \cap \textrm{aff}(\sigma))\big), \\
    \pi_{\sigma}&:= \pi_{\textrm{aff}(\sigma)},
\end{align*}
 where $\pi_{\textrm{aff}(\sigma)}$ was defined after Definition \ref{def:T-H}. By Lemma \ref{lem:local-equality}
\begin{multline*}
\sT_{\cC_1} \wedge \sT_{\cC_2}{_{|_{\Log^{-1}(B)}}} = w_{\sigma_1} w_{\sigma_2} \mathbbm{1}_{\Log^{-1}(B)} \\ \int_{(x_1,x_2)\in S^1(\sigma_1)\times S^1(\sigma_2)}[\pi_{\sigma_1}^{-1}(x_1)] \wedge [\pi_{\sigma_2}^{-1}(x_2)]\, d\mu_{\sigma_1}(x_1) \otimes d \mu_{\sigma_2}(x_2).
\end{multline*}
Transversality of the fibres implies
$$
[\pi_{\sigma_1}^{-1}(x_1)] \wedge [\pi_{\sigma_2}^{-1}(x_2)] = [\pi_{\sigma_1}^{-1}(x_1) \cap \pi_{\sigma_2}^{-1}(x_2)].
$$
By Proposition~\ref{prop:int-numb-tori}, we have 
\[
\kappa = [\Z^d : \Z^d_{\mathrm{aff}(\sigma)} + \Z^d_{\mathrm{aff}(\sigma')}]
\]
distinct intersection points covering 
\[
\Log^{-1}(a) \simeq S^1(\sigma_1) \times S^1(\sigma_2).
\]
When \((x_1,x_2) \in \Log^{-1}(a) \simeq S^1(\sigma_1) \times S^1(\sigma_2)\) vary with respect to the normalised Haar measure, these \(\kappa\) points cover \(S^1(\sigma_1) \times S^1(\sigma_2) \simeq (S^1)^d\) with speed \(\kappa\). As a result,
\[
\begin{aligned}
&\int_{(x_1,x_2)\in S^1(\sigma_1)\times S^1(\sigma_2)} [\pi_{\sigma_1}^{-1}(x)] \wedge [\pi_{\sigma_2}^{-1}(x')] \, d\mu_{\sigma_1}(x) \otimes d\mu_{\sigma_2}(x') \\
&\quad = \int_{y \in (S^1)^d} \kappa \, [\pi_{\sigma_1 \cap \sigma_2}^{-1}(y)] \, d\mu_{\sigma_1 \cap \sigma_2}(y).
\end{aligned}
\]
This proves (b).

\vskip 2mm
To deduce (c), let \(\sigma = \sigma_1 \cap \sigma_2\) be a \((p + q - d)\)-dimensional cell in the intersection. Assume that \(0 \in \sigma\), by a translation, and let \(L := \mathrm{aff}(\sigma)^{\perp}\). By Proposition~\ref{prop:assoc-commut},
\[
(\sT_{\cC_1} \wedge \sT_{\cC_2}) \wedge \sT_L = \sT_{\cC_1} \wedge (\sT_{\cC_2} \wedge \sT_L).
\]
Note that \([N : N_{\sigma} + N_L] = 1\). Assume that \(w_{\sigma_1} = w_{\sigma_2} = 1\). As a result, if the multiplicity of \(\sT_{\cC_1} \wedge \sT_{\cC_2}\) at \(\sigma\) equals \(\kappa\), then the multiplicity of \((\sT_{\cC_1} \wedge \sT_{\cC_2}) \wedge \sT_L\) at the origin is also \(\kappa\). We also have:
\[
[N : N_{\sigma_1} + N_{\sigma_2 \cap L}] = 
\left[ \frac{N}{N_{\sigma}} : \frac{N_{\sigma_1} + N_{\sigma_2 \cap L}}{N_{\sigma}} \right] = 
\left[ \frac{N}{N_{\sigma}} : N_{\sigma_1 \cap L} + N_{\sigma_2 \cap L} \right] = 
[N : N_{\sigma_1} + N_{\sigma_2}].
\]

As a consequence, the intersection multiplicity induced on \(\sigma\) by \(\sT_{\cC_1} \wedge \sT_{\cC_2}\) equals the intersection multiplicity in Definition~\ref{def:stable-int}.
\vskip 2mm
To prove (d) note that if $\cC_1 + \epsilon b$ is the translation of the the tropical variety, where $ b\in \R^d$ and $\epsilon \in \R_{\geq 0},$ then $(e^{\epsilon b})^* \sT_{\cC_1}= \sT_{\cC_1 + \epsilon b}.$
Moreover, we have the SP-convergence of currents with continuous superpotentials. 
$$
(e^{\epsilon b})^* \sT_{\cC_1} \lto \sT_{\cC_1} \, ,\qquad \text{as $\epsilon \to 0.$}
$$
Therefore, by Theorem \ref{thm:wedge-cont-pot}, 
\begin{equation}\label{eq:current-perturb}
(e^{\epsilon b})^* \sT_{\cC_1} \wedge \sT_{\cC_2}= \sT_{\cC_1+ \epsilon b} \wedge \sT_{\cC_2}   \lto \sT_{\cC_1} \wedge \sT_{\cC_2}\, , \qquad \text{as $\epsilon \to 0.$} 
\end{equation}
Considering the support, we obtain the Hausdorff limit
$$
 \lim  \supp{(e^{\epsilon b})^* \sT_{\cC_1} \wedge \sT_{\cC_2}} \supseteq  \supp{\sT_{\cC_1} \wedge \sT_{\cC_2} } .
$$
We now note that for all $\epsilon,$ the number of intersection points in $(\cC_1+\epsilon b) \cap \cC_2$ is uniformly bounded by the number of $p$-dimensional cells in $\cC_1$ and $q$-dimensional cells in $\cC_2,$ the Hausdorff limit of $(\cC_1 + \epsilon b) \cap \cC_2$ is also zero dimensional. Now, by definition of $\cC_1 \cdot \cC_2$ it suffices to show that 
$$
 \lim  \supp{(e^{\epsilon b})^* \sT_{\cC_1} \wedge \sT_{\cC_2}} = \supp{\sT_{\cC_1} \wedge \sT_{\cC_2} },
$$
for any fixed generic $b$. This is also easy. Let $a_{\epsilon} \in (\cC_1 + \epsilon b) \cap  \cC_2.$ Since the translation by $\epsilon b$ does not change the slopes of the cells, as
$a_{\epsilon} \to a,$  the multiplicity for all $a_{\epsilon}$ remains constant for $\epsilon >0$, therefore $ \lim (e^{\epsilon b})^* \sT_{\cC_1} \wedge \sT_{\cC_2}$ has a non-zero mass at $\Log^{-1}(a).$
\vskip 2mm 
For Part (e), first observe that any $\lim (\cC_1+ \epsilon b ) \cap \cC_2$ is obtained by a translation $\epsilon b$, as $\epsilon \to 0,$ of finitely many $(p+q-d)$-dimensional cells. Therefore, $\cC_1 \cdot \cC_2$ is also of dimension $p+q -d,$ and the SP-convergence readily implies that the limit is independent of generic $b.$ Now, it only remains to show that
$$
 \lim  \supp{(e^{\epsilon b})^* \sT_{\cC_1} \wedge \sT_{\cC_2}} = \supp{\sT_{\cC_1} \wedge \sT_{\cC_2} }.
$$
Let $\sigma$ be a $p+q-d$ dimensional cell in $\cC_1 \cdot \cC_2,$ and by translation, assume that $0\in \sigma.$ Let $L = \aff(\sigma)^{\perp}.$ By Proposition \ref{prop:slicing-CSP}, 
$$
\big( (e^{\epsilon b})^* \sT_{\cC_1} \wedge \sT_{\cC_2} \big) \wedge \sT_{L} \lto \big( \sT_{\cC_1} \wedge \sT_{\cC_2} \big) \wedge \sT_{L}, \quad \text{as $\epsilon\to 0.$ } 
$$
By Part (d), the left-hand-side has mass at the origin. This show (e). 
\vskip 2mm
To deduce Part (f), recall that in Equation (\ref{eq:current-perturb}) since we can choose $b$ generically and the previous discussion. The balancing condition is also deduced by the fact that $\sT_{\cC_1} \wedge \sT_{\cC_2}$ is closed and Theorem \ref{thm:closed-balanced}.
%
%
%
%
\end{proof}

\begin{remark}\label{rem:non-positive-int}
When $\cC$ is a tropical cycle which is not positively weighted, then there exist positively weighted tropical cycles $\cC_1$ and $\cC_2$ such that
    $$
    \cC = \cC_1 - \cC_2.
    $$
    Therefore, if $\Sigma$ is compatible with both $\cC_1$ and $\cC_2,$ then by preceding lemma, $\overline{\sT}_{\cC}$ has a continuous superpotential in $X_{\Sigma}.$ In such a case, the wedge product $\overline{\sT}_{\cC} \wedge \sS,$ for any current $\sS$ is well-defined. In particular, we can define the stable intersection of two non-positively weighted tropical cycles, however, in this case, we might not have the Hausdorff convergence of the supports as in the positive case due to the cancellations.  
\end{remark}

Let us recall a few notable theorems from \cite{McMullen-polytope-alg}, \cite{Fulton--Sturmfels}, and \cite{Jensen-Yu}, and add the tropical currents into the picture. 

\begin{theorem}\label{thm:Q-alg}
   The assignment $ \cC \lmto \sT_{\cC} $ induces a $\mathbb{Q}$-algebra  isomorphism between
   \begin{itemize}
        \item [(a)] The $\mathbb{Q}$-algebra of tropical currents on $(\C^*)^d$ with the usual addition of currents and the wedge product.
        \item [(b)] The $\mathbb{Q}$-algebra generated by the tropical cycles in $\R^d$ with the natural addition  (Definition \ref{def:add-trop}) and stable intersection (Definition \ref{def:stable-int}).
       
   \end{itemize}
Further, considering the map $\Phi_m: (\C^*)^d \lto (\C^*)^d$,  $z\lmto z^m.$ We have the isomorphism of the following:
\begin{itemize}
\item [(a')] The $\mathbb{Q}$-algebra of $\Phi_m$-invariant tropical currents on $(\C^*)^d$ with the usual addition of currents and the wedge product.
\item [(b')] The $\mathbb{Q}$-algebra generated by the tropical fans in $\R^d$ with the natural addition  and stable intersection. 
     \item [(c')] The $\mathbb{Q}$-algebra generated by tropical hypersurfaces. 
      \item [(d')] The McMullen algebra of dimensional polytopes in $\mathbb{Q}^d$.
     \item [(e')] Chow group of the compact space obtained as a direct limit of all complete toric varieties. 
\end{itemize}
\end{theorem}
\begin{proof}
By Proposition \ref{prop:assoc-commut}, Items (a) and (a') form a $\mathbb{Q}$-algebra. In Item (b), the fact that the tropical cycles in $\R^d$ generate an algebra is the content of \cite{Jensen-Yu}*{Theorem 2.6}, which can also be deduced from Theorem \ref{thm:main-int} and Proposition \ref{prop:assoc-commut}, then the isomorphism of (a) and (b) is also clear. 
\vskip 2mm
For the isomorphism between (a') and (b'), we observe that a tropical current of bidimension $(p,p)$, denoted $\sT_{\cC}$, is invariant under the map $\Phi,$ that is,
\[
m^{p-d} \, \Phi_m^* \sT_{\cC} = \sT_{\cC}
\]
if and only if $\cC$ is a tropical fan.

The isomorphism between (b') and (e') is discussed in \cite{Fulton--Sturmfels} and can also be deduced directly from \cite{BH}*{Theorem 4.7}. The isomorphism between (d') and (e') is proved in \cite{Fulton--Sturmfels}*{Theorem 4.2}, while that between (b') and (d') appears in \cite{Jensen-Yu}*{Theorem 5.1}. The equality of (b') and (c') is treated independently in \cite{Kazarnovski}*{Proposition 9} and \cite{Jensen-Yu}*{Corollary 5.2}.
\end{proof}

A nice generalisation of the isomorphism between items (a') and (d') to a homeomorphism in dimension 2, along with its applications to dynamical systems, is presented in \cite{DillerRoeder2025}.

\subsubsection{Calculating Intersection Multiplicities Using Monge-Amp\`ere Measures}
Using the equality of the supports in the previous section, we only need to prove the intersection multiplicities in the transversal case locally.

\subsubsection{The Real Monge--Amp\`ere Measures }
Let $\Omega\subseteq \R^d$ be an open subset and $u:\Omega \lto \R$ be a convex (hence continuous) function. The \emph{generalised gradient} of $u$ at $x_0\in \Omega$ is defined by 
$$
\nabla u(x_0) = \big\{\xi \in (\R^d)^*: u(x)-u(x_0) \geq \langle \xi , x- x_0 \rangle, \text{ for all $x\in \Omega$}   \big\}.
$$
In the above, $\langle ~ , ~ \rangle$ denotes the standard inner product in $\R^d,$ and $(\R^d)^*$ is the dual. The real Monge-Amp\`ere measure associated to a convex function $u$ on a Borel set $E \subseteq \Omega,$ is given by
$$
\textrm{MA} [u](E)  =  \mu \big( \bigcup_{y \in E} \nabla u (y)  \big),
$$
where $\mu$ is the Lebesgue measure on $(\R^d)^*.$
\vskip 2mm
It is interesting that for the tropical polynomials, one can compute the associate real Monge--Amp\`ere measures explicitly. Recall that, for any tropical polynomial, there is a natural subdivision of its Newton polytope which is dual to the tropical variety of it. See Figure for an example and \cite{brugale-a-bit, Maclagan-Sturmfels} for details.

\begin{lemma}[\cite{yger}*{Page 59},\cite{GPS}*{Proposition 2.7.4}]\label{lem:real-MA-dec}
    Let $\fq:\R^d \lto \R$ be a tropical polynomial associated tropical variety $\cC = V_{\textrm{trop}}(\fq),$ one has
    $$
     \textrm{MA}[\fq] = \sum_{a \in  \cC(0)} \textrm{Vol} \big(\{ a\}^* \big) \delta_{a},
    $$
where $\cC(0)$ is the $0$-dimensional skeleton of $\cC,$ and $\{ a\}^*$ is the dual of the vertex $a\in \cC(0).$     
\end{lemma}
\noindent
A detailed discussion of the preceding theorem can be also found in \cite{Babaee}. 

\subsection{Polarisation}
For $n$ convex functions $u_1, \dots , u_d: \R^d \lto \R,$ their \emph{mixed Monge-Amp\`ere measure} is defined by 
$$
\widetilde{\text{MA}}[u_1, \dots , u_d] = \frac{1}{d!}  \sum_{k=1}^{d} \, \sum_{1 \leq j_1 < \dots < j_k \leq d} (-1)^{d-k} \, {\text{MA}}[u_{j_1}+ \dots + u_{j_k}].
$$
Recall that this is how the \emph{mixed volume} of $n$ convex bodies can be defined from the $n$-dimensional volume. Moreover, it is easy to check that for a convex function $u:\R^d \lto \R,$ $\text{MA}[u] = \widetilde{\text{MA}}[u, \dots , u].$
\vskip 2mm
It is customary at this point to prove the Bernstein--Khovanskii--Kushnirenko Theorem, which is simply obtained by taking the total mass from Lemma \ref{lem:real-MA-dec}.
\begin{proposition}
Let $\fq, \fq_1, \dots , \fq_d:\R^d \lto \R$ be tropical polynomials. We have the following facts:
    \begin{itemize}
        \item [(a)] $\textrm{MA}[\fq](\R^d) = \textrm{Vol}_d (\Delta_{\fq}),$ where $\Delta_{\fq}$ is the Newton polytope of $\fq.$
        \item [(b)] (Tropical Bernstein--Khovanskii--Kushnirenko Theorem) $\widetilde{\text{MA}}[\fq_1, \dots , \fq_d](\R^d) = \widetilde{\text{Vol}} (\Delta_{\fq_1}, \dots , \Delta_{\fq_d}),$ where $\widetilde{\text{Vol}}$ is the mixed volume.    
    \end{itemize}
\end{proposition}
\begin{proposition}\label{cor:int_mult_MA}
    Assume that $\alpha_i, \beta_i \in \Z^d$ for $i=1, \dots , d.$ Let $\fq_i = \max \{ \langle \alpha_i, x \rangle , 
   \langle \beta_i , x \rangle \}$ be $n$ tropical polynomials. Then, 
    $$
    d! \, \widetilde{\text{MA}}[\fq_1, \dots , \fq_d] = \kappa \, \delta_{0},
    $$
    where $\kappa$ is given by the volume \emph{zonotope} of the Minkowski sum of the vectors $\sum_{i=1}^d [ \alpha_i - \beta_i].$ In addition, these multiplicities coincide with the intersection multiplicities in Definition \ref{def:stable-int} for $d$ tropical hypersurfaces. 
 \end{proposition}
    
\begin{proof}
 Note that $\Delta_{\fq_i} $ is the line segment between $\alpha_i$ and $\beta_i.$ Moreover, in the definition of 
 $\widetilde{\text{MA}}[\fq_1, \dots , \fq_d]$ only $\textrm{Vol}( \sum_{i=1}^d [\alpha_i - \beta_i])$ possibly has a non-zero $n$-dimensional volume.  Finally, the origin is the only $0$-dimensional cell of the tropical variety of polynomial $\fq_1 + \dots + \fq_d$ , if and only if, $\{\alpha_1-\beta_1, \cdots ,  \alpha_d-\beta_d\}$ forms a linearly independent set.  Therefore, $d! \, \textrm{MA}[\fq_1 + \dots + \fq_d]= \kappa \, \delta_{0}.$ To prove the second statement, we first claim that any Borel subset $\Omega \subseteq \R^n,$ 
 $$
 \int_{\Log^{-1}(U)} \sT_{\Trop(\fq_1)} \wedge \dots \wedge  \sT_{\Trop(\fq_d)}= d! \, \widetilde{\text{MA}}[\fq_1, \dots , \fq_d ](\Omega).
 $$
 To see this, recall that ${\sT_{\Trop(\fq_i)}} = dd^c [\fq_i \circ \Log],$ by Theorem \ref{thm:hyper-ddc} and Remark \ref{rem:trop-hyper-def}. Further, by \cite{Rash}
 $$
\int_{\Log^{-1}(U)} \big(dd^c [\fq \circ \Log]\big)^d =  d! \, {\text{MA}}[\fq] (U).
 $$
 Polarising both sides proves the claim. This implies that the tropical intersection multiplicities coincide with the multiplicities in this proposition for hypersurfaces. 
\end{proof}








\section{Slicing Tropical Currents}
\begin{proposition} \label{prop:slice}
Let $\cC$ be a $p$-dimensional positively weighted tropical cycle in $\R^d$ with $p\geq 1.$ Assume that  $S\subseteq (\C^*)^d$ is an algebraic hypersurface with transversal intersection with $\sT_{\cC}.$ Then, $[S] \wedge \sT_{\cC} $ is admissible and it is a closed positive current of bidimension $(p-1,p-1)$ given by
$$
[S] \wedge \sT_{\cC} = \sum_{\sigma \in \cC} w_{\sigma}\int_{x \in S_{N(\sigma)}} \mathbbm{1}_{\Log^{-1}(\sigma^{\circ})} \big[ S \cap \pi_{\mathrm{aff}(\sigma)}^{-1}(x) \big] \ d\mu(x).
$$
\end{proposition}
\begin{proof} 
The idea of the proof is similar to that of \cite[Proposition 4.11]{BH}. Let $f$ be a polynomial with vanishing $S$ in $(\CC^*)^d$. Assume that $\Log^{-1}(\sigma^{\circ})\cap S \neq \varnothing,$ for a $p$-dimensional cell $\sigma \in \cC,$ then for each fibre, $\pi_{\sigma}^{-1}(x):= \pi_{\aff(\sigma)}^{-1}(x),$ the transversality assumption allows for the application of the Lelong--Poincar\'e formula to deduce
$$
    \quad dd^c \big(\log|f|\mathbbm{1}_{\Log^{-1}(\sigma^{\circ})}\big[\pi^{-1}_{\sigma}(x) \big] \big)  = \mathbbm{1}_{\Log^{-1}(\sigma^{\circ})} [ S \cap \pi^{-1}_{\sigma}(x) \big] + \sR_{\sigma}(x), 
$$
where $\sR_{\sigma}(x)$ is a $(p-1, p-1)$-bidimensional current. The support of $\sR_{\sigma}(x)$ lies in the boundary of $\Log^{-1}(\sigma),$ as $\sR_{\sigma}(x)$ is the difference of two currents that coincide in any set of form $ \Log^{-1}(B),$ where $B\subseteq \R^d$ is a small ball with
  $$
  B\cap \sigma' = \varnothing, \quad \text{for a $p$-dimensional cell  $\sigma'\in \cC, \sigma'\neq \sigma$ },
  $$
and both vanish outside $\Log^{-1}(\sigma).$
Integrating along the fibers, and adding for all $p$-dimensional cells $\sigma \in \cC,$ we obtain
\begin{equation*}
[S] \wedge {\sT_{\cC}}  = \sum_{\sigma \in \cC} w_{\sigma}\int_{ x \in S_{N(\sigma)}} \mathbbm{1}_{\Log^{-1}(\sigma^{\circ})} \big[ S \cap \pi_{\mathrm{aff}(\sigma)}^{-1}(x) \big] \ d\mu(x) + \sR_{\cC},
\end{equation*}
where $\sR_{\cC}$ is $(p-1,p-1)$-dimensional current. We claim that $\sR_{\cC}$ is \emph{normal}, \textit{i.e.} $\sR_{\cC}$ and $d\sR_{\cC}$ have measure coefficients; $\sR_{\cC}$ is a difference of two normal currents, where the first current $[S] \wedge {\sT_{\cC}}$ is a positive closed current, and the second current is an addition of normal pieces. Moreover, the support of $\sR_{\cC}$ is a subset of $S$ as it is a difference of two currents that both vanish outside $S.$ As a result, the current $\sR_{\cC}$ is supported on $S\, \cap \, \bigcup_{\sigma} \partial \Log(\sigma) .$ This set is a real manifold of Cauchy--Riemann dimension less than $p-1$, therefore by Demailly's first theorem of support the normal current $\sR_{\cC}$ vanishes; see also the discussion following \cite[Proposition 4.11]{BH}.
\end{proof}

\begin{corollary}
Let $H \subseteq \R^d$ be a rational plane of dimension $r\geq d-p$ and $A:= a+ H,$ a translation of $H$ for $a \in \R^d.$ Assume also that  
$\cC \subseteq \R^d$ is a tropical variety of dimension $p$ that intersects $A$ transversely. Then 
$$
[(e^{-a})T_{H\cap \Z^d}] \wedge \sT_{\cC}
$$
can be viewed as a tropical current of dimension $p-(d-r)$ in the complex subtori $T^A:= (e^{-a})\, T_{H \cap Z^d}\subseteq (\C^*)^d.$
\end{corollary}
\begin{proof}
Note that the hypothesis implies that the intersection $T^A \cap \,\pi_{\mathrm{aff}(\sigma)}^{-1}(x)$ is transversal for any ${x \in S_{N(\sigma)}}.$  By translation, it is sufficient to prove the statement for $a=0.$ By the preceding theorem, 
$$
[T^A] \wedge \sT_{\cC} = \sum_{\sigma \in \cC} w_{\sigma}\int_{x \in S_{N(\sigma)}} \mathbbm{1}_{\Log^{-1}(\sigma^{\circ})} \big[ T^A \cap \pi_{\mathrm{aff}(\sigma)}^{-1}(x) \big] \ d\mu(x).
$$
The sets  $ T^A \cap \pi_{\mathrm{aff}(\sigma)}^{-1}(x)$ can be understood as the toric sets in $T^A$. We can conclude since each every $d\mu$ associated to each $\sigma$ is a Haar measure. 
\end{proof}

\begin{theorem}\label{thm:slicing}
Let $M \subseteq (\mathbb{C}^*)^{d-p}$ and $N \subseteq (\C^*)^{p}$ be two bounded open subsets such that $N$ contains the real torus $(S^1)^p$. Let $\pi : M \times N \lto M$ be the canonical projection. Let $\sT_n$ be a sequence of positive closed $(p, p)$-bidimensional currents on $M \times N$ such that $\overline{\supp{\sT_n}} \cap (M \times \partial \overline{N}) = \varnothing$. Assume that $\sT_n \lto \sT$ and $\supp{\sT} \subseteq M \times (S^1)^p$. Then we have the following convergence of slices
\[
\langle \sT_n | \pi| x\rangle \lto \langle \sT | \pi| x\rangle \quad \text{for every } x \in M.
\]
Note that all the above slices are well-defined for all $x \in M$.
\end{theorem}

\begin{proof} 
Since all the currents $\sT_n$ and $\sT$ are horizontal-like, the slices are well-defined, and we prove that the slices have the same cluster value. Let $\sS$ be any cluster value of $\langle \sT_n | \pi| x\rangle.$   Note that such $\sS$ always exists by Banach--Alaoglu  theorem. As both measures $\sS$ and $\langle \sT | \pi| x\rangle,$ are supported $\{x\}\times (S^1)^p$ to prove their equality, it suffices to prove that they have the same Fourier coefficients. By Theorem \ref{thm:Berteloot-Dinh}, we have 
   $$
    \langle \sS , \phi \rangle \leq  \langle \sT | \pi| x\rangle (\phi),
    $$
    for every plurisubharmonic function $\phi$ on $\C^d,$ and the mass of $\sS$ coincides with the mass of $\langle \sT | \pi| x\rangle.$ Now, note that if  $\phi$ is pluriharmonic, then $-\phi$ and $\phi$ are plurisubharmonic. As a result, 
    $$
    \langle \sS , \phi \rangle =  \langle \sT | \pi| x\rangle (\phi),
    $$
    for every pluriharmonic function. Recall that if $ f$ is a holomorphic function, then $\text{Re}(f)$ and $\text{Im}(f)$ are pluriharmonic. We now consider the elements of the Fourier basis $f(\theta)= \exp{2 \pi i \langle \nu , \theta \rangle }$ for $\nu \in \Z^d.$ Then we have the equality 
    $$
    \langle \sS , f \rangle =  \langle \sT | \pi| x\rangle (f).
    $$
This implies that the Fourier measure coefficients of both $\sS$ and $\langle \sT | \pi| x\rangle $ coincide. 
\end{proof}

\begin{lemma}\label{lem:radon}
    Let $\cC\subseteq \R^d$ be a tropical variety of dimension $p$, and $L$ be a rational $(d-p)$-dimensional plane such that $L$ is transversal to all the affine extensions $\text{aff}(\sigma)$ for $\sigma \in \cC.$
    Assume that  $\sT$ is a positive closed current of bidimension $(p,p)$ on a smooth projective toric variety $X_{\Sigma}$ compatible with $\cC + L$ such that $\supp{\sT} \subseteq \supp{\sT_{\cC}}.$ Further, for all $a \in \R^d,$
    $$
    \overline{\sT}_{L+a } \wedge \sT = \overline{\sT}_{L+a} \wedge \overline{\sT}_{\cC}.
    $$
  Then $\sT = {\sT_{\cC}}$ in $T_N.$
\end{lemma}

\begin{proof}
Let us first remark that $\text{rec}(L+a) = \text{rec}(L)$ for all $a \in \R^d$ and therefore, all $\sT_{a+ L}$ are compatible with $X_{\Sigma}$ and have a continuous superpotenrial in $X_{\Sigma}$ and as a result, all the above wedge products are well-defined. By Demailly's second theorem of support~\cite[III.2.13]{DemaillyBook1}, there are measures $\mu^{\sT}_{\sigma}$ such that 
$$
\sT = \sum_{\sigma}\int_ {x \in S(\Z^d \cap H_{\sigma})} \mathbbm{1}_{\Log^{-1}(\sigma^{\circ})}\big[\pi_{\sigma}^{-1}(x)\big] d\mu^{\sT}_{\sigma}(x).
$$
By repeated application of  Proposition \ref{prop:slice}, 
$$
\sT_{L} \wedge \sT =  \sum_{\sigma}\int_{(x,y) \in S(\Z^d \cap H_{L})\times  S(\Z^d \cap H_{\sigma}) } \big[\pi_{H}^{-1}(x)\cap \pi_{\sigma}^{-1}(y) \big]\,d \mu_{L}(x) \otimes \mu^{\sT}_{\sigma}(y).
$$
Applying both sides of the equality $\sT_{L} \wedge \sT = \sT_{L} \wedge \sT_{\cC}$ on test-functions of the form
\begin{equation*}\label{eq:fourier-forms}
\omega_{\nu}= \exp(-i\langle \nu , \theta \rangle )\rho(r) 
\end{equation*}
where $\rho: \R^d \to \R$ is a smooth function with compact support of $r \in \R^d$ and $\theta \in [0,2\pi)^d,$ and $\nu \in \Z^d,$ completely determines the Fourier coefficients of $\mu^{\sT}_{\sigma}$ which have to coincide with the normalised Haar measures multiplied by the weight of $\sigma$, \textit{i.e.}, $\mu_{\sigma}^{\sT}= w_{\sigma} \mu_{\sigma}.$
\end{proof}

Note that any subtorus of $(\C^*)^d,$ can be understood as a fibre of a tropical current. We have the following slicing theorem. 
\begin{proposition}\label{prop:slicing}
    Let $\cC \subseteq \R^d$ be a tropical variety and $A\subseteq \R^d$ be a rational hyperplane intersecting $\cC$ transversely. Let $\Sigma$ be a fan compatible with $\cC+ A.$ Assume that ${\overline{\sS}_n}$ is a sequence of positive closed currents on $X_{\Sigma},$ and denote by $\sS_n$ the restriction to $T_N.$ Further, assume that 
\begin{itemize}
\item [(a)] $\overline{\sS}_n \lto \overline{\sT}_{\cC};$
\item [(b)] $\supp{\overline{\sS}_n}  \lto \supp{\overline{\sT}_{\cC}},$
\end{itemize}
then
$$
\lim_{n\to \infty} \big({\sS_n} \wedge [T^A] \big) = \sT_{\cC} \wedge [T^A],
$$
as currents on $T_N \subseteq X_{\Sigma}.$
\end{proposition}

\begin{proof}
Assume that $L\subseteq \R^d$ is an $(d-p-1)$-dimensional affine plane intersecting all $\text{aff}(\sigma)$ for all $\sigma \in \cC \cap A$ transversely. Then, on a projective smooth toric variety $X_{\Sigma'}$ compatible with $\cC+L+ A$ the tropical currents $\overline{\sT}_{a + L},$ $a\in \R^d$ have continuous superpotentials. Therefore, by Proposition \ref{prop:slicing-CSP}, we have
$$
\lim_{n\to \infty} \big(\overline{\sS}_n \wedge \overline{\sT}_{a+L} \big) =  \overline{ \sT}_{\cC}  \wedge \overline{\sT}_{a+L}.
$$
Now, for any $x \in \cC \cap L \cap A,$ let $B\subseteq \R^d$ containing $x$ be a bounded open set containing only $x$ as an isolated point of the intersection. By a translation we can assume that $x=0.$ Consider the isomorphism
$$\xi: (\C^*)^d \xrightarrow{~\sim~} T_{\Z^d /(\Z^d \cap A)} \times T_{\Z^d \cap A},$$ 
and let $\pi_1$ and $\pi_2$ be the respective projections. Note that $\pi_1^{-1}(1) = T^A.$ We now set
\begin{align*}
    U &:= \pi_1\circ\xi\big( \Log^{-1}(U) \cap \supp{\sT_{\cC} \wedge \sT_{a+L}} \big),\\
    V &:=  \pi_2 \circ \xi\big( \Log^{-1}(U) \cap  T^A \big), \\
    \sT_n &:= \xi_* \big(\sS_n \wedge \sT_{a+L} \big), \text{ in $T_N$}, \\
    \sT &:= \xi_* \big( \sT_{\cC}  \wedge \sT_{a+L}\big).
\end{align*}
Note that $\sT_{\cC}$ are horizontal-like as in the setting of Theorem \ref{thm:slicing}.  Assumption (b) now implies that  $\sT_n$ for a large $n,$ is also horizontal-like. Thus we obtain
$$ 
\lim_{n\to \infty} \big(\sS_n  \wedge [T^A] \big)\wedge \sT_{a+L}= \sT_{\cC} \wedge [T^A] \wedge \sT_{a+L},
$$
for every $a.$ We now deduce the convergence on $T_N$ by Lemma \ref{lem:radon}. 
\end{proof}

\begin{theorem}\label{thm:limit-closure}
Let $\mathcal{C} \subseteq \mathbb{R}^d$ be a tropical variety, and let $B = A_1 \cap \dots \cap A_k \subseteq \mathbb{R}^d$ be a complete intersection of rational hyperplanes $A_1, \dots, A_k$. Assume that $\mathcal{C}$ and $B$ intersect transversely. Let $\Sigma$ be a smooth, projective fan compatible with $\mathcal{C} + B$, and let $(\overline{\mathscr{S}}_n)$ be a sequence of positive closed currents on $X_{\Sigma}$. Suppose that:
\begin{itemize}
    \item[(a)] $\overline{\mathscr{S}}_n \lto \overline{\mathscr{T}}_{\mathcal{C}}$;
    \item[(b)] $\supp{\overline{\mathscr{S}}_n} \lto \supp{\overline{\mathscr{T}}_{\mathcal{C}}}$.
\end{itemize}

Then, the following limit holds in the smooth projective toric variety $X_{\Sigma}$:
\[
\lim_{n \to \infty} \left(\overline{\mathscr{S}}_n \wedge [\overline{T}^{B}]\right) = \overline{\mathscr{T}}_{\mathcal{C}} \wedge [\overline{T}^B].
\]
\end{theorem}
We need a simple observation:
\begin{lemma}\label{lem:closure}
 Let $U \subseteq \C^d$ be an open subset and $D$ be an analytic subset of $\C^d.$  Assume that we have the convergence of closed positive currents $\sV_n \lto \sV$ in $U\setminus D,$ and $\sV_n$'s and $\sV$ have a uniformly bounded local masses near $D.$ Further, assume that for any cluster value $\sW$ of the sequence $\{\overline{\sV}_n\}_n,$  we have
 \begin{itemize}
     \item [(a)] $\supp{\sW}\subseteq \supp{\overline{\sV}},$
     \item [(b)] $\supp{\overline{\sV}} \cap D$ has the expected Cauchy--Riemann dimension.
 \end{itemize} 
 Then, 
 $$\overline{\sV}_n \lto \overline{\sV}.$$
\end{lemma}

\begin{proof}
$\sW -\overline{\sV}$ is a positive closed current with the Cauchy--Riemann dimension less than or equal to $p,$ therefore, it must be zero by Demailly's first theorem of support  \cite{DemaillyBook1}*{Theorem III.2.10}.
\end{proof}

\begin{proof}[Proof of Theorem \ref{thm:limit-closure}]
 Applying Theorem \ref{thm:transverse-fiber} (or \cite{Osserman-Payne}*{Proposition 3.3.2} to each fibre of $\overline{\sT}_{\cC}$ separately), we obtain 
 $\supp{\overline{\sT}_{\cC}} \cap \overline{T}^B \cap [D_\rho]$ has the expected Cauchy--Riemann dimension $p-k-1.$ By Demailly's first theorem of support  \cite{DemaillyBook1}*{Theorem III.2.10},
 $$\overline{\sS}_{\cC} \wedge [\overline{T}^B] = \overline{\sT_{\cC} \wedge [T^B]}. $$ 
 By assumption $\overline{\sS}_n  \lto \overline{\sT}_{\cC}$ and $\supp{\sT_n}  \lto \supp{\overline{\sT}_{\cC}}.$ The observation in Lemma \ref{lem:slice-supp}, 
$$
 \lim_{n \to \infty} \supp{\overline{\sS}_n \wedge [\overline{T}^B]} \subseteq \supp{\overline{\sT}_{\cC} \wedge [\overline{T}^B]}.
$$
Therefore, any cluster value of $\overline{\sS_n \wedge [T^B]} \leq \overline{\sS}_n \wedge [\overline{T}^B] $ has a support in $\supp{\overline{\sT}_{\cC} \wedge [\overline{T}^B]}.$ Now, we set
\begin{itemize}
\item [(a)] $\sV_n := \sS_n \wedge [T^B],$
\item [(b)] $\sV := \sT_{\cC} \wedge [T^B],$
\item [(c)]  $\sW$ a cluster value of $\overline{\sT_n \wedge [{T}^B]}.$
\end{itemize}
A repeated application of Proposition \ref{prop:slicing} for $B = A_1 \cap \dots \cap A_k,$ we are in the situation of Lemma \ref{lem:closure}, and  conclude. 
\end{proof}
 \begin{lemma}\label{lem:diag-int}
   Let $X_{\Sigma}$ be a smooth projective toric variety, and $\bar{\Delta} \subseteq X_{\Sigma}\times X_{\Sigma}$ be the diagonal. Let $\sS$ and $\sT$ be two positive currents on $X.$ Then, for any ray $\rho \in \Sigma,$
$$
\supp{\sS} \cap \supp{\sT} \cap D_{\rho} \subseteq X_{\Sigma}
$$
has a Cauchy--Riemann dimension $\ell$, if and only if,
$$
\supp{\sS \otimes \sT} \cap \bar{\Delta} \cap D_{(0,\rho)} \subseteq X_{\Sigma} \times X_{\Sigma},
$$
has a Cauchy--Riemann dimension  $\ell,$ where $D_{(0,\rho)}$ is the toric invariant divisor corresponding to the ray $(0, \rho)$ in $\Sigma \times \Sigma.$
 \end{lemma}
\begin{proof}
    The fan of $X_{\Sigma} \times X_{\Sigma}$ is $\Sigma \times {\Sigma},$ we have that   $D_{(0, \rho)}\simeq X_{\Sigma} \times D_{\rho}$ and the assertion follows. 
\end{proof}
\begin{theorem}\label{thm:int-general}
    Let $\cC_1, \cC_2\subseteq \R^d$ be two tropical cycles intersecting properly. Assume that $X_{\Sigma}$ is a smooth toric projective variety compatible with $\cC_1 +\cC_2.$ If moreover, for two sequence of positive closed currents $\overline{\sV}_n$ and $\overline{\sW}_n$ we have 
\begin{itemize}
    \item [(a)] $\overline{\sW}_n \lto \overline{\sT}_{\cC_1},$ and $\overline{\sV}_n \lto \overline{\sT}_{\cC_2};$
   \item [(b)] $\supp{\overline{\sW}_n} \lto \supp{\overline{\sT}_{\cC_1}},$ and  $\supp{\overline{\sV}_n} \lto \supp{\overline{\sT}_{\cC_2}};$ 
 \item [(c)] For any large $n,$ $\overline{\sW}_n \wedge  \overline{\sV}_n$  is well-defined,
\end{itemize}
 then 
$$
\overline{\sW}_n \wedge \overline{\sV}_n \lto \overline{\sT}_{\cC_1} \wedge \overline{\sT}_{\cC_2}, \quad \text{as $m\to \infty.$ }
$$
 
\end{theorem}
\begin{proof}
For two closed currents $\sS$ and $\sT$ on $X_{\Sigma}$ we naturally identify 
$$\sS \wedge \sT = \pi_* \big(\sS \otimes \sT  \wedge  [\bar{\Delta}]\big),$$
where $\pi: X_{\Sigma} \times X_{\Sigma} \lto X_{\Sigma}$ is the projection onto the first factor. Let $\sW_n$ and $\sV_n$ be the restriction of $\overline{\sW}_n$ and $\overline{\sV}_n$ to $T_N \subseteq X_{\Sigma}.$ 
We can define the tropical current $\sT_{\cC} := \sT_{\cC_1} \otimes \sT_{\cC_2}.$ By Theorem \ref{thm:closed-balanced}, $ \cC_1 \times \cC_2 := \cC $ is a balanced tropical variety. Let us denote the diagonal $\Delta_{\R} \subseteq \R^d \times \R^d.$ It is not hard to see that the transversality of $\cC_1$ and $\cC_2$ is equivalent to the transversality of the intersection 
$$(\cC_1 \times \cC_2) \cap \Delta_{\R} \subseteq \R^d \times \R^d.$$ 
Now, fixing the coordinates $\big((z_1, \dots ,z_d ), (z'_1, \dots , z'_d)\big) \in (\C^*)^d \times (\C^*)^d,$ the diagonal of $ (\C^*)^d \times (\C^*)^d$ is given by the complete intersection of the tori $z_i = z'_i,$ $i=1, \dots , d.$  Therefore, by refining $\Sigma \times \Sigma$ to a fan compatible also with $\Delta_{\R},$ and setting $B= \Delta$ in Proposition \ref{prop:slicing}, we obtain
$$
\sT_n \wedge [\Delta] \lto \sT_{\cC} \wedge [\Delta], \quad {\emph{i.e.}, }~  \sW_n \wedge \sV_n \lto \sT_{\cC_1} \wedge \sT_{\cC_2},
$$
in the open torus $T_N \times T_N.$ Now, we can employ the compatibility of $\cC_1 + \cC_2$ and $X_{\Sigma}$ together with Lemma \ref{lem:diag-int}, imply that for any ray $\rho \in \Sigma,$
$$
\supp{\overline{\sT}_{\cC_1} \otimes \overline{\sT}_{\cC_2}} \cap  [\bar{\Delta}] \cap D_{\rho}
$$
has the expected Cauchy--Riemann dimension, and finally conclude by Lemma \ref{lem:closure}.
\end{proof}

\textcolor{red}{}

\section{Dynamical Tropicalisation and Intersections}
\subsection{Dynamical tropicalisation with a non-trivial valuation}
Recall that for a field $\mathbb{K},$  $\nu:\mathbb{K}\lto \R \cup \{\infty \},$ is called a valuation if it satisfies the following properties for every $a, b \in \mathbb{K}$: 
\begin{itemize}
    \item [(a)] $\nu(a) = \infty$ if and only if $a=0;$
    \item [(b)] $\nu(ab) = \nu(a) + \nu(b);$
    \item [(c)] $\nu(a+b) \geq \min \{\nu(a) , \nu(b) \}.$
\end{itemize}
A valuation is called \emph{trivial}, if the valuation of any non-zero element is 0. For an element $a \in \mathbb{K},$ we denote by $\bar{a}$ its image in the residue field. We are interested in the case where $\mathbb{K}= \C((t)),$ is the field of \emph{formal Laurent series} with the variable $t,$ with the usual valuation. That is, for $g(t) = \sum_{j\geq k} a_j t^j,$ with $a_k \neq 0,$ the valuation equals the minimal exponent $\nu(g) = k \in \mathbb{Z}.$  
\begin{definition}
 \begin{itemize}
     \item [(a)] Let $f= \sum_{\alpha\in A}c_{\alpha} z^{\alpha} \in \mathbb{K}[z^{\pm 1}]$, be a Laurent polynomial in $d$ variables, where $A\subseteq \mathbb{Z}^d$ is a finite subset.  The tropicalisation of $f$ with respect to $\nu,$ 
         \begin{align*}
            \trop_{\nu}(f)&: \R^d \lto \R, \\
              x &\mapsto \max \{-\nu(c_{\alpha}) + \langle x, \alpha \rangle \}.
       \end{align*}


\item [(b)] Let $I \subseteq \mathbb{K}[z^{\pm 1}]$ be an ideal. The tropical variety associated to $I$, as a set, is defined as 
$${\text{Trop}_{\nu}}(I) := \bigcap_{f \in I} \text{Trop}(\trop_{\nu}(f)),$$
where $\text{Trop}(\trop_{\nu}(f))$ is the set of points where $\trop_{\nu}(f)$ is not differentiable; see Remark \ref{rem:trop-hyper-def}. 

\item [(c)] For an algebraic subvariety of the torus $Z \subseteq {(\mathbb{K}^*)}^d,$ with the associated ideal $\mathbb{I}(Z)$, the tropicalisation of $Z$, as a set, is $\Trop_{\nu}(Z):= \Trop_{\nu}(\mathbb{I}(Z)).$ 
\item [(d)] In all the situations above, $\trop_0$ denotes the tropicalisation with respect to the trivial valuation.
 \end{itemize}    
\end{definition}
We need to relate a non-trivial valuation to the trivial valuation. Compare to \cite{BJSST}*{Lemma 1.1}.
\begin{lemma}\label{lem:int-1}
Consider the ideal $I \subseteq \C[t^{\pm 1}, z^{\pm 1}] \stackrel{\iota}{\hookrightarrow}
 \C((t))[z^{\pm 1}].$ Assume that $(u,x)$ are the coordinates in $ \R\times \R^d$.  Then, we have the following equality of sets
$$
\Trop_0(I)  \cap \{u= -1 \} = \Trop_{\nu}(\iota(I)),
$$
where $\nu$ is the the usual valuation in $\C((t)).$ In other words, the tropicalisation of $I$ as an ideal in $\C[t^{\pm 1},z^{\pm 1}]$ with respect to the trivial valuation intersected with $\{u= -1\}$ coincides with the tropicalisation of $I =\iota(I)$ with respect to the usual valuation in $\C((t)).$ 
\end{lemma}
The proof of the lemma becomes clear with the following example.
\begin{example}
Let $$f(x,t) = 4(t^3+ t^{-1})z_1 z_2+ (1+ t + t^2) z_1.$$
Then, the tropicalisation of $f\in \C[t^{\pm 1}, z^{\pm 1}],$ with respect to the trivial valuation equals:
$$
\trop_0(f)= \max\big\{ \max \{3u+x_1+ x_2, -u + x_1+ x_2\}, \max \{x_1, u+x_1 , 2u+x_1 \} \big\}.
$$
Letting $u:=-1,$ $\trop_0(f)(-1,x)= \max\{ 1 + x_1+ x_2, x_1\}.$
The latter equals $\trop_{\nu}(f)$ as an element of $\C((t))[z^{\pm 1}].$
\end{example}

\begin{proof}[Proof of Lemma \ref{lem:int-1}]
   If $f$ is a monomial in $\C[t][z],$ then it is clear that
   $$
\trop_0(f)(-1,x) =  \trop_{\nu}(\iota(f)).
   $$
    Therefore, we have the equality for any polynomial in $f \in \C[t,z].$ To prove the main statement, note that 
\begin{align*}
     \Trop_{\nu}(\iota(I)) & = \bigcap_{f\in \iota(I)} \Trop(\trop_{\nu}(f))\\ & = \bigcap_{f\in I}\big(\Trop(\trop_{0}(f)) \cap \{ u= -1\}\big) \\ &= \Trop_0(I)  \cap \{u= -1 \}.  
\end{align*}

\end{proof}
\begin{remark}\label{rem:Bergman}
    Bergman in \cite{Bergman}, shows that for an algebraic subvariety $Z\subseteq (\C^*)^d,$ one has 
$$
\lim_{t \to \infty} \Log_{t}   (Z) \subseteq \Trop_{0}(\mathbb{I}(Z)),
$$
and he conjectured the equality. This conjecture was later proved by Bieri and Groves in \cite{Bieri-Groves}. More precisely, Bieri and Grove prove that $\lim \Log_{t}   (Z) \cap (S^1)^d$ is a polyhedral sphere of real  dimension equal to (the complex dimension) $\dim(Z) -1.$ Therefore, the fan $\lim \Log_{t}   (Z)$ is a cone over their spherical complex. See also \cite{Maclagan-Sturmfels}*{Theorem 1.4.2}.
\end{remark}
 \begin{remark}
     The above lemma is related to the results of Markwig and Ren in \cite{Markwig--Ren}. They considered the tropicalisation of an ideal $J \subseteq R[[t]][z],$ where $R$ is the ring of integers of a discrete valuation ring $\mathbb{K}$, which is non-trivially valued. To obtain finiteness properties, however, the authors consider the associated tropical variety in the half-space $\R_{\leq 0} \times \R^d.$ Note that such a variety is almost never balanced. The authors also prove that for an ideal $I \subseteq \mathbb{K}[z],$ the tropicalisation of the natural inverse image $\pi^{-1}I \subseteq R[[t]][z]$ with respect to trivial valuation, intersected with $\{u=-1 \}$ equals $\trop_{\nu}(I);$ \cite{Markwig--Ren}*{Theorem 4}. 
 \end{remark}

Let us also recall the main result of \cite{Dyn-trop}. While preparing this article, we realised that a version of this theorem, presented without proof, appeared in \cite{Kazarnovski}*{Theorem 1}.

\begin{theorem}\label{thm:dyn_trop_toric}
 Let $Z\subseteq (\CC^*)^d$ be an irreducible subvariety of dimension $p,$ and $\overline{Z}$ be the closure of $Z$ in a compatible smooth projective toric variety $X_{\Sigma}$. Define $\Phi_m: X_{\Sigma}\longrightarrow X_{\Sigma}$ to be the unique continuous extension of 
\begin{align*}
    (\C^*)^d & \longrightarrow (\C^*)^d, \\
    z &\lmto z^m.
\end{align*}
Then,
$$
\frac{1}{m^{d-p}}\Phi_m^*[\overline{Z}] \longrightarrow \overline{\mathscr{T}}_{\Trop_0(Z)}, \quad \text{as $m\to \infty$},
$$
where  $\overline{\mathscr{T}}_{\Trop_0(Z)}$ is the extension by zero of $\mathscr{T}_{\Trop_0(Z)}$ to $X_{\Sigma}.$ Moreover, the supports also converge in Hausdorff metric. 
\end{theorem}

Note that since the limit of a sequence of closed currents is closed, the above theorem implies that $\trop_0(Z)$ can be equipped with weights to become balanced. Note that the compatibility is in the following sense of Tevelev and Sturmfels: 
\begin{theorem}
\begin{itemize}
  \item [(a)]  The closure $\overline{Z}$ of $Z$ in $X_{\Sigma}$ is complete, if and only if, $\Trop_0(Z) \subseteq |\Sigma|;$ see \cite{Tevelev}. 
  \item [(b)] We have $|\Sigma|= \Trop_0(Z)$, if and only if, for every $\sigma \in \Sigma$ the intersection $ \cO_{\sigma}\cap \overline{Z}$ is non-empty and of pure dimension $p- \dim(\sigma);$ see \cite{Sturmfels-Tevelev}.
\end{itemize}
\end{theorem}

\begin{theorem}\label{thm:dyn-trop-nont}
    Let $I \subseteq \C[t^{\pm 1}, z^{\pm 1}]$ be an ideal with the associated $(p+1)$-dimensional algebraic variety $W= \mathbb{V}(I) \subseteq (\C^*)^{d+1}.$ Assume that the projection onto the first coordinate $\pi_1: W \lto \C^*$ is a surjective and Zariski-closed map. Further,  $\pi_1(W_{\text{sing}}) \subsetneq \C^*$, the projection of singular points of $W$ is a proper subset of $\C^*$. If we denote the fibers by $W_{t} := \pi_1^{-1}(t),$ then
    \begin{itemize}
        \item [(a)]  $\frac{1}{m^{d-p}}\Phi_m^* [W_{e^m}] \lto \sT_{\Trop_{\nu}(I)}, \quad \text{as $m \to \infty$ },$ in the sense of currents in $\sD_{p}((\C^*)^d)$, and we have identified $\iota(I)= I.$
    
    \item [(b)]  $\Trop_{\nu}(I)$ can be equipped with weights to become balanced.
    \item [(c)] $\lim \supp{\frac{1}{m^{d-p}}\Phi_m^* [W_{e^m}]} = \supp{\sT_{\Trop_{\nu}(I)}}.$
    \item [(d)] On a toric variety $X_{\Sigma}$ compatible with $\trop_0(W)+ \{u=-1 \},$
    $$
    \frac{1}{m^{d-p}}\Phi_m^* [\overline{W}_{e^m}] \lto \overline{\sT}_{\Trop_{\nu}(I)}, \quad \text{as $m \to \infty$.}
    $$
    \end{itemize}
\end{theorem}

For the proof, we need the following: 
\begin{lemma}\label{lem:trans}
Let $W\subseteq (\C^*)^{d+1}$ be a $(p+1)$-dimensional subvariety, such that the projection onto the first factor, $\pi_1:(\C^*)^{d+1} \longrightarrow \C^*  $ is surjective and a Zariski closed morphism, and the projection of the singular points are a proper subset: $\pi_1(W_{\text{sing}})\subsetneq \C^*$. Then, for a sufficiently large $|t_0|>>0$
$$
[W_{t_0}] = [\pi_{1}^{-1}(t_0)] = [\{t=t_0\}] \wedge [W].
$$
\end{lemma}
\begin{proof}
Let us first fix an ideal associated to $I = \mathbb{I}(W) = \langle f_1, \dots , f_k\rangle \subseteq \C[t^{\pm 1}, z^{\pm 1}].$ At any regular point $w\in W_{\text{reg}},$ $T_w W = \ker J(f)(w) = \big(\frac{\partial f_i}{\partial z_j}(w) \big)_{k\times (d+1)}$ is of dimension $p+1.$
We consider the \emph{critical set}
$$
C=\big\{w\in W: \dim \ker \begin{pmatrix}
\nabla_w t \\
Jf(w) 
\end{pmatrix} \geq p+1 \big\},
$$
where $\nabla_w t$ is the gradient of $t$ at $w$. We have that $\nabla_w t = e_1,$ the first element of the standard basis of the $\C$-vector space $\C^{d+1}.$ As a result, to have the rank of $\begin{pmatrix}
e_1 \\
Jf(w) 
\end{pmatrix}$ less than or equal to $(d+1)-(p-1)$ we must have all the minors of size $ (d-p+1)\times (d-p+1)$ equal to zero, which are the defining equations for $C.$  The set Zariski-closed set $C$ contains all the points such that either
\begin{itemize}
    \item $w \in W$ is regular and the intersection $W \cap \pi_{1}^{-1}(\pi_1(w))$ is not transversal. Note that in this case, $\text{Image}(Jf(w))\subseteq \text{ker}(e_1).$
   \item $w \in W$ is singular, and therefore $\ker (Jf(w)) > p+1.$
\end{itemize}
Since $\pi_1$ is a closed map, if we observe that $\pi_1(C) \subsetneq \C^*$, then $\pi_1(C)$ is a finite set and we conclude. Note that $\pi_1(W_{\text{reg}} \cap C) = \C^*,$ implies that $W$ is contained in $\{t=t_0 \}$ for some $t_0 \in \C^*,$ which is not the case as $\pi_1: W \lto \C^*$ is surjective. Further. $\pi(W_{\text{sing}}) \subsetneq \C^*$ by assumption of the lemma, and we conclude. 
\end{proof}

\begin{proof}[Proof of Theorem \ref{thm:dyn-trop-nont}]
  By the preceding lemma, and the fact that $\Phi_m$ preserves transversal intersection, we have 
  $$ \frac{1}{m^{d-p}}\Phi_m^* [W_{e^m}] = \frac{1}{m^{d-(p+1)}} \Phi_m^*[W] \wedge \frac{1}{m} \Phi_m^*[\{t= e^m\}],$$
  for a large $m.$ Since $\Trop_0(W)$ is a fan, it is transversal to the plane $\{u =-1\} \subset \R^{d+1}$. In consequence, we can use Proposition \ref{prop:slicing} to write
  $$
  \lim \frac{1}{m^{d-p}}\Phi_m^* [W_{e^m}] = \big(\lim  \frac{1}{m^{d-(p+1)}} \Phi_m^*[W]   \big) \wedge \big(\lim \frac{1}{m} \Phi_m^*[\{t= e^m\}] \big)
  $$
 By Theorem \ref{thm:dyn_trop_toric}, restricted to $(\C^*)^{n+1},$ and the fact that we used $\Log = (-\log|\cdot| , \dots , -\log|\cdot|)$ in the definition of tropical currents,  the above limits yield
  $$
    \lim \frac{1}{m^{d-p}}\Phi_m^* [W_{e^m}] = \sT_{\Trop_0(W)} \wedge \sT_{\{u\, = \, -1\}}.
  $$
  Applying Theorem \ref{thm:main-int} and Lemma \ref{lem:int-1}, we obtain the equality. 
  \vskip 2mm
  For the assertion (b), note that the limit $\sT_{\Trop_{\nu}(I)}$ is a closed current and Theorem \ref{thm:closed-balanced} implies that $\Trop_{\nu}(I)$ is naturally balanced. To observe (c), note that (a) implies 
  $$
  \lim \supp{\frac{1}{m^{d-p}}\Phi_m^* [W_{e^m}]}  \supseteq \supp{\sT_{\Trop_{\nu}(I)}}.
  $$
  However, because of transversality,   $\supp{\sT_{\Trop_{\nu}(I)}} = \supp{\sT_{\Trop_0(W)}} \cap \supp{ \sT_{\{u\, = \, -1\}}}.$ At the same time, 
 $$
  \lim \supp{\Phi_m^*[W_{e^m}]} = \lim \supp { \Phi_m^*[W] } \cap \supp{\Phi_m^*[\{t= e^m\}]}.
 $$
Moreover, for the Hausdorff limit of sets $\lim (A_i \cap B_i) \subseteq (\lim A_i) \cap (\lim B_i).$ This implies
$$
  \lim \supp{\Phi_m^*[W_{e^m}]} \subseteq \supp{\sT_{\Trop_0(W)}} \cap \supp{ \sT_{\{u\, = \, -1\}}},
$$
  which implies (c). Now, (d) is implied by Theorem \ref{thm:limit-closure}.
\end{proof}

In the setting of the previous theorem, a generalisation of Bergman's theorem (see Remark \ref{rem:Bergman}) asserts that 
$$
\Log_{t} (W_t) \lto \Trop(I), \quad \text{as $t \lto \infty,$}
$$
where $\Log_t$ is the logarithm with base $t.$ This theorem can be understood as a counterpart of Lemma \ref{lem:int-1} for tropicalisation with $\Log.$ This generalisation was finally proved by Jonsson in \cite{Jonsson}, and we can now deduce a sequential analogue: 
\begin{corollary}
  In the setting of the previous theorem 
  $$
  \frac{1}{m} \Log (W_{e^m}) \lto \Trop_{\nu}(I), \quad \text{as $m \to \infty,$}
  $$
  in the Hausdorff metric in compact subsets of $\R^d.$
\end{corollary}
\begin{proof}
Note that for any variety $Z \subseteq (\C^*)^{d+1},$ as a set $\Log (\Phi_m^{-1}(Z)) = \frac{1}{m}\Log(Z).$ Therefore, 
$$
\Log~ \supp{  \frac{1}{m^{d-p}}\Phi_m^* [W_{e^m}]} = \frac{1}{m} {\Log (W_{e^m})}. 
$$
The assertion now follows from Theorem \ref{thm:dyn-trop-nont}(c) and continuity of $\Log$ with respect to the Hausdorff metric on compact sets. 
\end{proof}

In light of Theorem \ref{thm:dyn_trop_toric}, the following result is a very special case of Theorem \ref{thm:int-general}. It is analogous to \cite{BJSST}*{Lemma 3.2} and \cite{Osserman-Payne}*{Theorem 1.2}.

\begin{theorem}\label{thm:proper-int}
Assume that $W$ and $V$ are two algebraic subvarieties of $(\C^*)^d$ with respective dimensions $p$ and $q,$ with $p+q \geq d.$ Assume that $\Trop_0(V)$ and $\Trop_0(W)$ intersect transversely, then 
$$
\frac{1}{m^{2d-(p+q)}}\Phi^*_{m} \big([W]  \wedge [V] \big) \lto \sT_{\Trop_{0}(W) \cdot \Trop_{0}(V)}, \quad \text{as $m\to \infty.$}
$$
Moreover, 
$$\sT_{\Trop_0(W \cap V)} \leq \sT_{\Trop_{0}(W) \cdot \Trop_{0}(V)}.$$ 
In particular, the corresponding induced weights for $\tau \in \Trop_0(W \cap V),$ is less than or equal to the weight of $\tau$ induced by $\Trop_{0}(W) \cdot \Trop_{0}(V).$

\begin{example}
To see that the inequality in the previous theorem can be strict, let us consider the subvarieties of $(\C^*)^2,$ $W= \mathbb{V}(z_2-1)$ and $V= \mathbb{V}(z_2 - z_1^2 - 1 )$. We have that
\begin{align*}
   \Trop_0(W) &=  \Trop (\max\{ x_2, 2x_1, 0\}), \\
   \Trop_0(V) &= \Trop( \max \{x_2, 0 \}).
\end{align*}
In Example \ref{eg:stable-int}, we discussed that the stable intersection of these two cycles is the origin $(0,0) \in \R^2$ with multiplicity 2. Note that the (set-theoretic) intersection $W \cap V  = \{(0, 1)\},$ thus $\Trop_0 (W \cap V) = (0,0)$ with multiplicity 1, whereas $\Trop_0(W) \cdot \Trop_0(V) = (0,0)$ with multiplicity 2. We obtain 
$$
2\, \sT_{\Trop_0(W \cap V)} = \sT_{\Trop_{0}(W) \cdot \Trop_{0}(V)}.
$$
Note that this inequality, gladly, does not contradict \cite{Osserman-Payne}*{3.3.1} (a consequence of \cite{Fulton_int}*{8.2}),  which the scheme-theoretic intersection $W \cap V$ is considered. 
\end{example}

\begin{proof}[Proof of Theorem \ref{thm:proper-int}]
    Assume that $X_{\Sigma}$ is a toric variety compatible with $\Trop_0(W)+ \Trop_0(V).$ We need to show that the hypothesis (a), $\dots$, (d) of Theorem \ref{thm:int-general} for $\overline{\sW}_m= {m^{p-n}}\Phi^*_m[\overline{W}]$ and $\overline{\sV}_m= {m^{q-n}}\Phi^*_m[\overline{V}]$ is satisfied. Note that hypotheses (a) and (b) are implied by Theorem \ref{thm:dyn_trop_toric}. The hypotheses (c) and (d) for $\overline{\sW}_1$ and $\overline{\sV}_1$ are a result of \cite{Osserman-Payne}*{Proposition 3.3.2}. The hypotheses (c) and (d) for any $m$ are implied by the fact that $\Phi_m^* \big([\overline{W}] \wedge [\overline{V}]\big) =\Phi_m^*[\overline{W}] \wedge \Phi_m^*[\overline{V}].$ Therefore, 
    $$
    \overline{\sW}_m \wedge \overline{\sV}_m \lto \overline{\sT}_{\Trop_{0}(W)} \wedge \overline{\sT}_{\Trop_{0}(V)}, \quad \text{as $m \to \infty,$}
    $$
    and the first assertion is obtained by restricting to $T_N \subseteq X_\Sigma$ and Theorem \ref{thm:main-int}. To see the second assertion, note that by Theorem \ref{thm:dyn_trop_toric},
    $$
    \frac{1}{m^{2d-(p+q)}}\Phi^*_{m} [W \cap V ]  \lto \sT_{\Trop_0(W\cap V)}, \quad \text{as $m \to \infty.$}
    $$
    However, as currents $[W \cap V] \leq [W] \wedge [V]$ since the right-hand side might produce multiplicities on the intersection. Applying $\Phi_m^*$ to both sides preserves this inequality.  Therefore, for every $m,$
    $$
    \Phi^*_{m} [W \cap V ]  \leq \Phi_m^*[W] \wedge \Phi_m^*[V]. 
    $$
   We can conclude by taking the limit $m \lto \infty.$

\end{proof}
\end{theorem}
The tropical version of the following was observed in various places \cite{Osserman-Payne, Maclagan-Sturmfels}, which does not need to assume the proper intersection of tropicalisations. 
  
\begin{theorem}\label{thm:averaging}
    Let $W$ and $V$ be two smooth algebraic subvarieties of $(\C^*)^d$, and for $0<\epsilon <1,$ let $U_{\epsilon}((S^1)^d)$ be an $\epsilon$-neighbourhood of  $(S^1)^d.$ Then,
    $$
    \frac{1}{m^{2d-(p+q)}} \int_{(t_1, t_2)\in   U_{\epsilon}((S^1)^{2d})}\Phi_m^*\big([t_1 V] \wedge [t_2 W]\big)\, d\nu(t_1) \otimes d\nu(t_2) \longrightarrow \sT_{\Trop_0(V)} \wedge \sT_{\Trop_0(W)},
    $$
    as $m \to \infty$. Here, the $d\nu$ is the normalised Lebesgue measures on $U_{\epsilon}((S^1)^d)$. Moreover, when $V$ and $W$ are smooth the integrand in the left-hand-side can be replaced by $\Phi_m^*[t_1 V \cap t_2 W].$
\end{theorem}
\begin{proof}
    By a classical result of Kleiman \cite{Kleiman}, the intersection of generic translations $t_1 V \cap  t_2 W$ is proper and transversal in case $V$ and $W$ are smooth. Since $\Phi_m^*$ preserves proper intersections and transversality, we can separate the above integrand into $\Phi_m^*[t_1 V] \wedge \Phi_m^*[t_2 W]$. Using polar coordinates $(r_i, \theta_i)$ for $t_i$, we have
    $$
    \int_{t_1 \in U_{\epsilon}((S^1)^d)} \Phi^*_m[t_1 V] \, d\nu(t_1) = \int_{t_1 \in U_{\epsilon}((S^1)^d)} t_1^{1/m} \, \Phi^*_m[V] \, d\nu(t_1).
    $$
    Here, with abuse of notation, we choose the $m$-th root $t_1^{1/m}$ to be the first root of $t^{1/m}.$ We can also obtain a similar equation for $t_2 W.$ Further, 
    $$
       \big( \Phi_m^*[t_1 V] \otimes \Phi_m^*[t_2 W] ) \wedge [\Delta] = \big(\Phi_m^*[ V] \otimes \Phi_m^*[ W] \big) \wedge  (t_1)^{-1/m} (t_2)^{-1/m}  [\Delta].
    $$   

    $$
    \int_{U_{\epsilon}(S^1)^{2n}} (t_1)^{-1/m} (t_2)^{-1/m}  [\Delta] = \int_{[(1-\epsilon) , 1+\epsilon]} \sT^{R_m}_{\Delta} dR,
    $$
     where $R_m = (|t_1|^{-1/m}, |t_2|^{-1/m}),$ and $\sT^{R_m}_{\Delta}$ is the tropical current associated with the diagonal $\Delta \subseteq \R^d \times \R^d,$ where the compact torus is rescaled to $R_m (S^1)^{2d}.$ In a toric variety compatible with $\Delta \subseteq \R^d \times \R^d,$ for any $R,$ $\overline{\sT^{R_m}_{\Delta}}$ has a continuous superpotential, and $\overline{\sT^{R_m}_{\Delta}}$ is SP-convergent to $\overline{\sT}_{\Delta}.$ Using Proposition \ref{prop:slicing-CSP}, Theorem \ref{thm:tensor-conv}, and restricting yields:
     $$
      \frac{1}{m^{2d-(p+q)}} \big( \Phi_m^*[t_1 V] \otimes \Phi_m^*[t_2 W] )\wedge  \int_{[(1-\epsilon) , 1+\epsilon]} \sT^{R_m}_{\Delta} dR \lto \big(\sT_{\Trop_0(V)} \otimes \sT_{\Trop_0(W)} \big) \wedge \sT_{\Delta},
     $$
     as $m\to \infty.$ The latter equals the tropical current 
     $$
     \sT_{\big(\Trop_0(V) \times \Trop_0(W) \big)\cdot \Delta}, 
     $$
     which can be identified with $\sT_{\Trop_0(V)} \wedge \sT_{\Trop_0(W)}.$
\end{proof}

\bibliographystyle{alpha}
\bibliography{bibnew}

\end{document}